\documentclass[11pt]{amsart}
\usepackage{amsmath}
\usepackage{amssymb}
\usepackage{amsthm}
\usepackage{mathtools}
\usepackage{physics}
\usepackage{graphicx}
\usepackage{amsfonts}
\usepackage[shortlabels]{enumitem}
\usepackage[utf8]{inputenc}
\usepackage[margin=2.85cm]{geometry}
\usepackage[framemethod=TikZ]{mdframed}
\usepackage[bookmarksopen=true,bookmarksnumbered=true, bookmarksopenlevel=2, colorlinks=false, pdfstartview=FitH]{hyperref}
\usepackage{theoremref}
\usepackage{tikz-cd}
\usepackage{mathrsfs}
\usepackage{stmaryrd}

\usepackage{comment}
\usepackage{color}

\usepackage{subfiles}

\specialcomment{details}{%
\begingroup\color{green}}{%
\endgroup}

\usepackage{titlesec}
\usepackage{textcomp}

\titleformat{\section}[hang]
{\centering\large\scshape}
{\thesection.}{.5em}{}

\titleformat{\subsection}[runin]
{\normalfont\bfseries}
{\thesubsection.}{.5em}{}

\titleformat{\subsubsection}[runin]
{\normalfont\itshape}
{\thesubsubsection.}{.5em}{}

\titlespacing{\section}{0pc}{.5cm}{.5cm}
\titlespacing{\subsection}{0pc}{.8cm}{.5pc}

\theoremstyle{plain}
\newtheorem{theorem}{Theorem}[section]
\newtheorem{proposition}[theorem]{Proposition}
\newtheorem{lemma}[theorem]{Lemma}
\newtheorem{corollary}[theorem]{Corollary}

\theoremstyle{definition}

\newtheorem{example}[theorem]{Example}
\newtheorem{remark}[theorem]{Remark}

\newcommand{\tp}[1]{\prescript{t}{}#1}

\newcommand{\nm}[1]{\left\lVert#1\right\rVert}

\DeclareMathOperator{\Gal}{Gal}

\DeclareMathOperator{\PSL}{PSL}
\DeclareMathOperator{\SO}{SO}
\DeclareMathOperator{\OO}{O}

\DeclareMathOperator{\Sym}{Sym}
\DeclareMathOperator{\GL}{GL}

\DeclareMathOperator{\Sp}{Sp}

\DeclareMathOperator{\Aut}{Aut}

\DeclareMathOperator{\Span}{Span}

\DeclareMathOperator{\erfc}{erfc}
\DeclareMathOperator{\sgn}{sgn}

\DeclareMathOperator{\Nm}{Nm}

\DeclareMathOperator{\Iso}{Iso}

\DeclareMathOperator{\signature}{signature}

\def\hmath$#1${\texorpdfstring{{\rmfamily\textit{#1}}}{#1}}

\title[Linking numbers and special differential characters]{Linking numbers and special differential characters in arithmetic hyperbolic manifolds}
\author{Mads Christensen}
\address{Max-Planck-Institut für Mathematik, Vivatsgasse 7, 53111 Bonn, Germany.}
\email{christensen@mpim-bonn.mpg.de}

\begin{document}

\begin{abstract}
    We construct non-holomorphic Siegel modular forms whose coefficients encode linking numbers in arithmetic hyperbolic manifolds of simplest type. Our proofs rely on an improved understanding of the transgression forms appearing in Kudla--Millson theory, and along the way we establish some structural properties including in particular their expected current equations. 
\end{abstract}
\maketitle{}

\section{introduction}

In a previous paper \cite{christensen2024linking}, we investigated genus $2$ Siegel type $q$-series related to linking numbers of certain special geodesics of arithmetic hyperbolic $3$-folds. In this sequel we start by adding several improvements to the results of loc.cit.; firstly we consider arithmetic hyperbolic $n$-folds $M$ of simplest type for any odd $n$ and any basefield, secondly we consider all possible genus for the associated Hilbert--Siegel series, and thirdly we consider linking numbers involving a more general (fixed) closed totally geodesic cycle.

Related results appear in many papers \cite{zwegers2008mock,MR3874688,MR3862073,westerholt2016indefinite,MR3708083,MR3763408,MR3993802,MR3796415,MR4570161,MR4611672}, but to our knowledge \cite{MR4611672} is the only other paper that considers Siegel theta series of arbitrary genus. Similarly to multiple of the cited papers we rely on Kudla--Millson theta series to establish modularity of (non-holomorphic completions of) our $q$-series. The perspective that we try to push in this paper is that our genus $m$ series should be regarded as lifts of totally geodesic ($m-1$)-cycles (rather than of some $m$-chain with the geodesic cycle as boundary). We note that the restriction to odd $n$ in this paper is just for simplicity, and by using the metaplectic double cover our results should carry over to any $n$.



Another, but related, point of this paper is to give an explicit construction of transgression forms for special geodesic cycles. For orthogonal groups of arbitrary signatures, such forms were constructed in \cite{MR2097357} for genus $1$ by using Borcherds singular theta lift. Our method is instead based on the ideas of Mathei and Quillen, as in \cite{MR3783423,MR3935034}, which allows us to consider arbitrary genus. We prove that our transgression forms are sparks in the sense of \cite{MR1993742}, and hence they upgrade the special geodesic cycles to differential character classes in the sense of Cheeger and Simons \cite{MR488077,MR827262}. We then consider the natural pairing of these classes in differential cohomology and obtain a modularity result for $q$-series of such pairings.

\subsection*{Results}

We now explain the results in a little more detail. Let $(\mathbf{V},Q)$ be a quadratic space over a totally real field $F$, whose signature at the various real places of $F$ are 
\begin{equation}\label{eq:vst}
    (n,1),(n+1,0),\ldots,(n+1,0).
\end{equation}
Let $\mathbf{L}\subset\mathbf{V}$ be an even integral lattice. For the purpose of exposition we assume that $\mathbf{L}$ is unimodular. In the following we drop various variables from the notation which are irrelevant in the unimodular case.  

Associated with any torsion-free arithmetic subgroup $\Gamma\subset \SO^0(\mathbf{L})$ is an oriented arithmetic hyperbolic orbifold $M$ and certain special geodesic cycles $c(T)$ in $M$ indexed by totally positive definite symmetric matrices $T\in \Sym_m(F)$ \cite{MR651982,MR1079646}. The $m$ equals the codimension of $c(T)$ in $M$. 

For $m>1$, the special geodesic cycles of Kudla and Millson do not actually describe all totally geodesic cycles in $M$.\footnote{We somewhat loosely use special geodesic cycle both to refer to the cycles $c(T)$, but also any of their components.} For example, if $m=n-1$, then the special geodesic cycles consist of geodesics with trivial holonomy. In the $(n,m)=(3,2)$ case in \cite{christensen2024linking} this is the same as what we called hyperbolic geodesics. To obtain more general totally geodesic cycles, one may consider a quadratic space $(\mathbf{V}',Q')$ over a totally real number field $F'$ containing $F$ whose restriction to $F$ embeds in $\mathbf{V}$. Then an arithmetic hyperbolic manifold associated with $\mathbf{V}'$ admits a totally geodesic immersion in $M$. In the bulk of the paper we refer to such cycles as special of general type. 

For a null homologous geodesic cycle of general type $c$ and $T$, for which $c$ and $c(T)$ do not intersect, we consider the linking number $\iota(T;s)$ between $c$ and $c(T)$, depending on an auxiliary relative homology class $s\in H_m(M,c)$. In the bulk of the paper we explain how to extend the definition of $\iota(T;s)$ to the cases where $c$ and $c(T)$ do intersect. In higher dimensions this is somewhat technical and unfortunately we only prove that $\iota(T;s)$ is a rational number in general. For $\tau$ in the Hilbert--Siegel upper-half space over $F$, and denoting by $q^T$ the usual holomorphic exponential function of $\tau$, we are then interested in the series
\[
    I(\tau;s) = \sum_{T\in\Sym_m(F)_+} \iota(T;s) q^T
\]
($\Sym_m(F)_+$ denotes totally positive definite matrices). 
Generalizing \cite{christensen2024linking} we construct an explicit non-holomorphic function $G^\ast(\tau;c)$, and we prove
\begin{theorem}\label{thm:ai}
    The function $I(\tau;s)+G^\ast(\tau;c)$ transforms as a Hilbert--Siegel modular form of genus $m$, weight $\frac{n+1}{2}$, and level $1$.
\end{theorem}
Let us touch briefly on the proof Theorem \ref{thm:ai} and the construction of the function $G^\ast(\tau;c)$. The point is that $I(\tau;s)+G^\ast(\tau;c)$ equals the integral over (a chain representing $s$) of the  theta series 
\[
    \Theta_\varphi(\tau) = \sum_{T\in\Sym_m(F)} \Theta_\varphi(T,v)q^T
\]
associated with the Kudla--Millson Schwartz $\varphi$, where here $v$ denotes the imaginary part of $\tau$. As in \cite{christensen2024linking,MR4570161} the idea is to study the family of differential forms $\Theta(T,\varphi^0,vt)$, with a variable $t>0$. For now, let us stick to the case of $T$ positive definite with $c$ and $c(T)$ not intersecting. As $t\to\infty$ we check that the forms $\Theta(T,\varphi^0,vt)$ converge, as currents, towards the delta distribution of the cycle $c(T)$. 
Based on this and further calculations, we obtain the linking numbers $\iota(T;s)$ as the limit $\lim_{t\to\infty}\int_s \Theta(T,\varphi^0,vt)$. 

To understand the remaining contribution, we consider a different Schwartz forms $\psi$, and the associated family of forms $\Theta_\psi(T,v)$. The main property of $\psi$ is the transgression equation
\[
    \frac{d}{dt} \Theta_\varphi(T,vt) = d\Theta_\psi(T,vt) \frac{1}{t}.
\]
Hence based on the construction in \cite{MR3783423, MR3935034} we are led to consider the form
\begin{equation}\label{eq:itd}
    \int_1^\infty \Theta_\psi(T,vt)\frac{dt}{t}
\end{equation}
which turns out to be smooth away from $c(T)$ and locally integrable on all of $M$. By the transgression equation and Stokes theorem, the integral of \eqref{eq:itd} over $c$ serves as the $T$th Fourier coefficient of $G^\ast(\tau;c)$. This works the same when $T$ is not positive definite, where $c(T)$ can be treated as the empty cycle. Here, a very convenient feature of the signature condition \eqref{eq:vst} is that $\varphi$ and $\psi$ vanish on all linearly dependent tuples of vectors in $\mathbf{V}$, so that we don't have to deal with special cycles of lower codimension. In the last degenerate case with $T$ positive definite and where $c(T)$ and $c$ intersect, one has to be much more careful with the order of integration and convergence issues. This leads us to instead define the $T$th Fourier coefficient of $G^\ast(\tau;c)$ as $\int_1^\infty \int_c \Theta_\psi(T,vt)\frac{dt}{t}$.

To prove convergence of the series $G^\ast(\tau;c)$ we use the naturality properties of $\psi$, together with a see-saw argument based on \cite{christensen2024linking} and \cite{MR4863919}. As in \cite{christensen2024linking} this leads to explicit formulas for $G^\ast$.

Actually, $G^\ast(\tau;c)$ is equally well-defined if $c$ is not null-homologous, so one could ask if there exist other interesting $q$-series, that really only depend on $c$ and which become modular when adding $G^\ast(\tau;c)$. We suggest an answer related to the differential cohomology of $M$, which is based on a current relation for \eqref{eq:itd} that comes from the transgression equation. Let us write $\Theta_{\widetilde\psi}(T,v) := -\int_1^\infty \Theta_\psi(T,vt)\frac{dt}{t}$.\footnote{This notation is mostly suggestive as there is no genuine Schwartz function $\widetilde\psi$.} We define the current of a locally integrable form $\alpha$ by $[\alpha](\eta) = \int_M \eta\wedge\alpha$, and the derivative of a current $C$ by $dC(\eta)=C(d\eta)$. We prove:
\begin{theorem}\label{thm:tse}
    $d[\Theta_{\widetilde\psi}(T,v)] = [\Theta_\varphi(T,v)] - \delta_{c(T)}.$
\end{theorem}
This implies that the special cycles $c(T)$ can be upgraded to a class $\hat{c}(T,v)$ in the $(m-1)$th differential cohomology groups of $M$. When $T$ is not positive definite, $\hat{c}(T,v)$ is a smooth class \cite[Definition 1.10]{MR1993742}.

Letting $m' = n+1-m$, the differential cohomology groups have a natural pairing between degree $m-1$ and $m'-1$, valued in $\mathbb{R}/\mathbb{Z}$. For the special cycle classes, the presense of explicit representatives also give explicit representatives $\varsigma(T,v;\hat{c})\in \mathbb{R}$ for the pairing of two special cycle classes $\hat{c}(T,v)$ and $\hat{c}$, at least when the special cycles $c(T)$ and $c$ do not intersect. As with $\iota(T;s)$, the definition of $\varsigma(T,v;\hat{c})$ is extended to the case of $c$ and $c(T)$ intersecting and we consider the series
\[
    S(\tau;\hat{c}) := \sum_{T\in\Sym_m(F)} \varsigma(T,v;\hat{c}) q^T.
\]

\begin{theorem}
    The function $S(\tau;\hat{c})+(-1)^m G^\ast(\tau;c)$ transforms as a Hilbert--Siegel modular form of genus $m$, weight $\frac{n+1}{2}$, and level $1$.
\end{theorem}


\subsection*{Outlook}

Many of the ideas in this paper very plausibly admit extensions to quadratic spaces of any signatures. The signature $(n,1)$-case treated in this paper has the advantage that very explicit formulas are accessible, but we expect that general signatures and genus can also be treated with help of the Mathai--Quillen formalism. Presumably the technical issues concerning intersections of cycles are more challenging in the general case.

We think it is worth pointing out that in the case where $s$ is closed the methods of this paper illustrate another way to prove the main results of \cite{MR1079646} in the case of signature $(n,1)$ which circumvents technical coordinate-based calculation with the Fock-model for the Weil representation.

The series $I(\tau;s)$ and $S(\tau;\hat{c})$ are somewhat reminiscent of (parts of) the arithmetic theta functions appearing in Kudla's programme. For example, consider the (level 1) Shimura curve $X$ associated with an indefinite quaternion algebra $B$ over $\mathbb{Q}$. For every $t\in \mathbb{Z}_{>0}$ let $z(t)$ denote the Heegner divisor on $X$ of discriminant $t$. In \cite{MR2220359} it is explained how to upgrade $z(t)$ to cycles $\hat{z}(t,v)$, $t\in\mathbb{Z}$, $v>0$, in the arithmetic Chow group of (an integral model of) $X$, and it is proved that the generating series
\[
    \widehat\phi_1(\tau) := \sum_{t\in \mathbb{Z}} \hat{z}(t,v) q^t
\]
(valued in the arithmetic Chow group) transforms as a classical modular form of weight $3/2$. This involves studying the Arakelov height pairing $\langle\;,\;\rangle$ of $\hat\phi_1$ with the various different arithmetic Chow group elements. The Arakelov height pairing also admits a decomposition $\langle\;,\;\rangle = \langle\;,\;\rangle_f+\langle\;,\;\rangle_\infty$ into a finite and an archimedian part. The infinite part $\langle\widehat{\phi}_1(\tau),\hat{z}\rangle_\infty$, for $\hat{z}$ an auxiliary arithmetic Chow group Heegner cycle, is defined in a way that is very analogous to the series $S(\tau;\hat{c})$, with Green functions replacing the sparks in this paper, and the coefficients of $G^\ast(\tau;c)$ are very analogous to the values of Green functions at $z$. It is probably dubious to try to take more away from this analogy, as there is obviously no arithmetic intersection theory available in the current setting. Nevertheless, we find the purely archimedean analogy curious.



\subsection{Notation}
Throughout $F$ will denote a number field, which besides Section \ref{sec:5.1} will always be totally real. We let $\mathbb{A}$, $\mathbb{A}_f$, and $\mathbb{A}_\infty\cong F \otimes_\mathbb{Q} \mathbb{R}$, denote the ring of adeles, finite adeles, and infinite adeles, respectively. We let $\Sym_m(\mathbb{R})_+$ and $\Sym_m(\mathbb{A}_\infty)_+$ denote the (totally) positive definite $m\times m$ symmetric matrices. We let $\GL_m^+(\mathbb{R})$ and $\GL_m^+(\mathbb{A})$ denote the matrices with (totally) positive definite determinant.

\subsection{Acknowledgements}

We are grateful to Luis Garcia and Jens Funke for many insightful thoughts and ideas that helped form this project. In particular, we thank Jens Funke for suggesting the spark equation of Theorem \ref{thm:tse}. We also thank Claudia Alfes and Lennart Gehrmann for their questions concerning extensions to higher dimensions and Romain Branchereau, Edgar Assing, and Noam Kimmel for various helpful discussions. We thank Max Planck Institute for Mathematics in Bonn
for its hospitality and financial support.

We used AI to help with literature search, and we acknowledge that AI found the convenient integral identity \eqref{eq:aii} and suggested the short proof strategy in Lemma \ref{lem:ail}.


\section{Special cycles}

\label{sec:1}

\subsection{Symmetric spaces}

Let $F$ be a totally real number field and let $\sigma_1,\ldots,\sigma_r$ denote the archimedean places of $F$. Let $(\mathbf{V},Q)$ be a quadratic space over $F$, with $(\;,\;)$ the associated bilinear form for which $\frac{1}{2}(X,X) = Q(X)$. For $i=1,\ldots,r$ we let $V^{(i)}$ be the completion of $\mathbf{V}$ corresponding to $\sigma_i$. We assume that $V:=V^{(1)}$ has signature $(n,1)$ and that $V^{(i)}$ has signature $(n+1,0)$ for $i\neq1$, and we fix an orientation of $V$. It will turn out to be very convenient to identify $F$ with a subfield of $\mathbb{R}$ by means of $\sigma_1$, and similarly to identify $\mathbf{V}$ with a subspace of $V$.



Let
\[
    \mathcal{D} := \{z\in V: Q(z) =- 1 \}
\]
denote the two-sheeted hyperboloid in $V$. By construction, the tangent space at $z\in \mathcal{D}$ is 
\[
    T_z \mathcal{D} = \{x\in V: (x,z)=0\},
\]
so that $V= T_z \mathcal{D}\oplus \mathbb{R}z$. We orient $\mathcal{D}$ by declaring that the basis $Y_1,\ldots, Y_n\in T_z \mathcal{D}$ is positively oriented precisely if $Y_1,\ldots,Y_n,z$ is a positively oriented basis for $V$. The restriction of $Q$ and $(\;,\;)$ to the tangent bundle of $\mathcal{D}$ determines a positive definite metric. We fix a connected component $\mathcal{H}\subset \mathcal{D}$ which provides a model for hyperbolic $n$-space with this metric. 

Let $\mathbf{G}:=\Res_{F / \mathbb{Q}}\OO(\mathbf{V})$, and for $i=1,\ldots,r$ let $G^{(i)}= \OO(V^{(i)})$. Then 
\[
    \mathbf{G}(\mathbb{R}) \cong \prod_{i = 1}^{r} G^{(i)} \approx \OO(n,1)\times \prod_{i = 2}^{r} \OO(n + 1). 
\]
Let $G^0\subset \mathbf{G}(\mathbb{R})$ be the connected component of the identity. 
The obvious actions identify $G^{(1)}$ and the connected component of its identity with the isometry group of $\mathcal{D}$ and the group of orientation preserving isometries of $\mathcal{H}$, respectively. In particular $\mathbf{G}(\mathbb{R})$ and $G^0$ act properly on $\mathcal{D}$ and $\mathcal{H}$, respectively. 

Let $\Gamma\subset G^0\cap \mathbf{G}(\mathbb{Q})$ be an arithmetic subgroup. We denote by $M_\Gamma = \Gamma\backslash \mathcal{H}$ the associated real analytic hyperbolic orbifold  of dimension $n$. As is well known, when $\Gamma$ is torsion free $M_\Gamma$ is a genuine hyperbolic manifold with no orbifold points. We will often abuse notation and simply write $M$ rather than $M_\Gamma$.

Let $\Iso_{\mathbf{V}}$ denote the variety, defined over $F$, of isotropic lines in $\mathbf{V}$. 
Then $M_\Gamma$ is closed if and only if $\Iso_{\mathbf{V}}(F)$ is empty, that is, if $\mathbf{V}$ is anisotropic. This is always the case if $F \neq \mathbb{Q}$ due to the signature condition on $\mathbf{V}$.

\subsection{General special cycles}

Let $U\subset V$ be a subspace for which the restriction of $Q$ is positive definite. Let $m= \dim U$ and $ n := n-m$, so that $U^\perp$ is of signature $(p,1)$.  We let $\mathcal{D}_U = \mathcal{D}\cap U^\perp$ denote the two-sheeted hyperboloid in $U^\perp$. For $z\in \mathcal{D}_U$ we see that
\begin{equation}\label{eq:tsd}
    T_z \mathcal{D} = T_z \mathcal{D}_U \oplus U.
\end{equation}
This gives a trivialization of the normal bundle $N_{\mathcal{D}_U/\mathcal{D}}$ with fiber $U$. In particular, a choice of orientation of $U$ determines an orientation of $\mathcal{D}_U$ compatible with the orientation of $\mathcal{D}$ through \eqref{eq:tsd}. We let 
\[
    c_U : = \mathcal{D}_U \cap \mathcal{H},
\]
which is a totally geodesic complete submanifold, and hence (a model for) hyperbolic $p$-space. 

To obtain cycles in $M$ we consider the following data. Let $F'$ be a totally real extension of $F$, and let ${\sigma}'_{i,j}$, for $i=1,\ldots,r$ and $j=1,\ldots,s$, denote the archimedean places of $ F$ with $\sigma'_{i,j}$ over $\sigma_i$. We identify $F'$ with a subfield of $\mathbb{R}$ by means of $\sigma'_{1,1}$. Let $(\mathbf{V}',Q')$ be a quadratic space over $F'$ satisfying the same signature condition as $V$
\[
    \signature V'^{(i,j)} = 
    \begin{cases} 
        (n',1) & i = j = 1 \\
        (n' + 1,0) & \text{else}
    \end{cases} 
    \hspace{2cm} n' \leq n
\]
($V'^{(i,j)}$ denotes the completion of $\mathbf{V}'$ at $\sigma_{i,j}$). We consider an isometric embedding of quadratic spaces over $F$
\[
    \eta:\Res_{ F' / F} \mathbf{V}' \hookrightarrow \mathbf{V}.
\]
Let $\mathbf{U}\subset \mathbf{V}$ denote the orthogonal complement of the image of $\eta$, and let $U\subset V$ denote the orthogonal complement of $V':=V'^{(1,1)}$. Note that $U$ is \emph{not} the completion $U^{(1)}$ of $\mathbf{U}$ at $\sigma_1$, rather
\begin{equation}\label{eq:vtu}
    U = \bigoplus_{j\neq 1}  V'^{(1,j)} \oplus U^{(1)}.
\end{equation}
Let
\begin{align*}
    \mathbf{G}' &:= \Res_{ F' / \mathbb{Q}}\OO({\mathbf{V}'})\times \Res_{F /\mathbb{Q}}\OO(\mathbf{U}) \subset \mathbf{G},
    \\
    \Gamma' &:= \Gamma \cap \mathbf{G}'(\mathbb{Q}).
\end{align*}
Hence $\mathbf{G}'(\mathbb{Q})$, resp. $\Gamma'$, is the subgroup of elements $g\in\mathbf{G}(\mathbb{Q})$, resp. $\gamma\in\Gamma$, which stabilize  the image of $\eta$ and preserve $Q'$. 
\begin{lemma}
    With notation as above:
    \begin{enumerate}[label=(\alph*), font=\upshape]
        \item $\mathbf{G}'(\mathbb{Q})$, resp. $\Gamma'$, coincides with the setwise stabilizer of $U$ in $\mathbf{G}(\mathbb{Q})$, resp. $\Gamma$;
        \item the map $\Gamma'\backslash c_U\to M_\Gamma$ is proper.
    \end{enumerate}
\end{lemma}
\begin{proof}
(a): It is clear from \eqref{eq:vtu} that $\mathbf{G}'(\mathbb{Q})$ stabilizes $U$. 

Now, let $\gamma\in \mathbf{G}(\mathbb{Q})$ be an element stabilizing $U$. In general, if $K/F$ is any field extension with
$ F' \otimes_F K = \prod_\alpha  K_\alpha$ for field extensions $ K_\alpha /  F'$, then one has an orthogonal decomposition
\[
    (\Res_{ F' / F} \mathbf{V}') (K) = \bigoplus_{\alpha}\mathbf{V}'(K_\alpha).
\]
For $K=F'$ there is a particular summand $\mathbf{V}'(F')$ corresponding to the trivial extension of $ F'$. This summand is also equal to $\mathbf{V}( F')\cap V'=\mathbf{V}( F')\cap U^\perp$, and hence is stabilized by $\gamma$. On the other hand, we may embed $\mathbf{V}'( F')= (\Res_{ F' / F} \mathbf{V}')(F)$ in $(\Res_{ F' / F}\mathbf{V}')( F')$ by extension of scalars, and the resulting map 
\[
    \mathbf{V}'( F') \hookrightarrow(\Res_{ F' / F} \mathbf{V}')( F') \to \mathbf{V}'( F')
\]
preserves $Q'$.



(b): Standard.

\end{proof}


More generally, if $\Gamma_0'\subset \Gamma'$ is a finite index subgroup, then we refer to the map 
\[
    c'_0:\Gamma_0'\backslash c_U\to M_\Gamma
\]
as a special geodesic cycle of general type. If $\Gamma$ is torsion-free, then this is a totally geodesic immersion. We see that the normal bundle of $c'_0$ is isomorphic to the bundle associated to the representation of $\Gamma'_0$ on $U$. In particular, if the image of $\Gamma'_0$ in $\OO(U)$ is contained in $\SO(U)$, then $\Gamma'_0\backslash c_U$ acquires an orientation from the orientation of $c_U$. 

If $\Gamma'_0$ acts trivially on $U$, which is only possible if $ F' = F$, then we obtain a trivialization of the normal bundle of $c'_0$ with fiber $U$. This is the classical case of special cycles considered in for instance \cite{MR651982,MR1079646}.

\begin{remark}
    The results of \cite{MR5051992} suggest that all complete, finite volume, totally geodesic, immersed submanifolds of $M$ arise in the manner described above.
\end{remark}

\subsection{Geodesics}

The case $ n' =1$, in which $c_U$ is a geodesic, is particularly nice. Although it is not absolutely necessary for the later parts of this paper, we pause to explain how our setup relates with the more classical viewpoint of non-parabolic elements in $\Gamma$. 

We got 
\[
    \mathbf{G}'(\mathbb{R}) \cong \prod_{\substack{1\leq i\leq r \\ 1\leq j\leq s}} \OO(V'^{(i,j)}) \times \prod_{1\leq i\leq r} \OO(U^{(i)}) \approx \OO(1,1)\times \OO(2)^{rs - 1}\times \OO(n + 1 - 2s)^r.
\]
If $\mathbf{V}'$ is isotropic, then $\Gamma'$ is a finite group (acting trivially on $c_U$) and the image of $c_U$ in $M_\Gamma$ is an infinite geodesic connecting the two cusps corresponding to the rational isotropic lines of $\mathbf{V}'$. If $\mathbf{V}'$ is anisotropic then the image of $c_U$ is a closed geodesic. In that case, the image of $\Gamma'$ in $\OO( V')$ is either infinite cyclic or infinite dihedral. 




On the other hand, let $\gamma\in \Gamma$ be any element of infinite order stabilizing a geodesic line in $\mathcal{H}$. We then claim that $\gamma$ belongs to $ \Gamma'$ for some isometry $(\eta,\mathbf{V}')$. To see this, let the geodesic line be $c_U$ and let $\ell$ and $\widetilde{\ell}$ denote the two isotropic lines of $U^\perp$. These are eigenspaces for $\gamma$ with real algebraic eigenvalues $\lambda$ and $\lambda^{-1}$, respectively. The remaining eigenvalues (of $\gamma|_U\in\OO(U)$) are then either $\pm1$ or come in complex conjugate pairs with absolute value $1$. Let 
\[
    \lambda_1 = \lambda, \;\lambda_2,\;\ldots, \;\lambda_s
\]
denote the $\Gal(\overline{\mathbb{Q}} / F)$-conjugates of $\lambda$. The corresponding eigenspaces span a $F$-rational subspace, and we take $\mathbf{U}$ to be its orthogonal complement. Let 
\[
    F' : = F(\lambda + \lambda^{ - 1}), \hspace{1cm}  E' := F(\lambda)
\]
Then $ F'$ is a totally real extension of degree $s$ over $F$, as $\lambda_j+\lambda_j^{-1}\in \mathbb{R}$, and $E'$ is a quadratic extension of $F'$. Let $u\in\ell$ be an $E'$-rational basis vector, and let $x\mapsto \tilde x$ denote the nontrivial automorphism of $E'/F'$. Then we see that
\[
    U^\perp \cap \mathbf{V}(F') = \{\;xu + \tilde x\tilde u: x\in  E'\;\}.
\]
We thus take $\mathbf{V}'$ to be the two-dimensional quadratic space on $E'$ with the quadratic form $x\mapsto (u,\tilde u) \Nm_{ E' /  F'}(x)$. To define the isometry $\eta$, let $\alpha_j: E'\to F(\lambda_j)$ be the unique homomorphism over $F$ with $\alpha_j(\lambda)=\lambda_j$, and let $u_j$ and $\tilde u_j$ denote the image of $u$ and $\tilde u$ under the map $\mathbf{V}(E')\to \mathbf{V}(F(\lambda_j))$ induced by $\alpha_j$. Then we set
\[
    \eta(x) = \sum_{j = 1}^s \alpha_j(x) u_j + \alpha_j(\tilde x)\tilde u_j, \hspace{1cm} x \in  E'.
\]
It is not hard to see that $\gamma\in \mathbf{G}'(\mathbb{Q})$ for this pair $(\eta,\mathbf{V}')$.

\begin{example}
    Let $\mathbf{V}$ be the quadratic space over $F=\mathbb{Q}$ with the underlying space $\{X\in M_2(\mathbb{Q}): \tr(X)=0\}$ and quadratic form $Q(X)=-\det X$ of signature (2,1). The group $\PSL_2$ acts by isometries on $\mathbf{V}$ by the formula
    \[
        g\cdot X = gX g^{ - 1}, \hspace{1cm} g\in \PSL_2,\; X\in \mathbf{V},
    \]
    which gives an embedding of algebraic groups $\PSL_2\subset \mathbf{G}$. We have $\PSL_2(\mathbb{R}) = G^0$, and the standard arithmetic subgroup $\Gamma:=\PSL_2(\mathbb{Z})$. An isometry between the Poincar\'e upper half-plane and $\mathcal{H}$ is given by
    \[
        \{z =x + yi\in \mathbb{C}: x\in \mathbb{R}, y\in \mathbb{R}_{ > 0}\} \cong \mathcal{H}, \hspace{1cm} z\mapsto  \frac{1}{y}\begin{pmatrix} - x & |z|^2 \\ - 1 & x\end{pmatrix},
    \]
    and this intertwines the usual action of $\PSL_2(\mathbb{R})$ and that of $G^0$. We have an isomorphism of varieties
    \begin{equation}\label{eq:pli}
        \mathbb{P}^1\cong\Iso_\mathbf{V},\hspace{1cm} (a:b)\mapsto \Span \begin{pmatrix} -ab & a^2 \\ - b^2 & ab\end{pmatrix} 
    \end{equation}
    which is equivariant for the action of $\PSL_2$ on the left side and $\mathbf{G}$ on the right side.
    
    Let $\gamma\in \Gamma$. The element $\gamma$ stabilizes a geodesic line if and only if $|\tr\gamma|>2$, that is, if $\gamma$ is hyperbolic. In that case, let $a,\tilde a\in \mathbb{R}$ be the two fixed points of $\gamma$ and
    $\varepsilon,\varepsilon^{-1}$ the corresponding eigenvalues of $\gamma$, when acting on $\mathbb{Q}^2$, for the eigenvectors $(a,1)$ and $(a',1)$ respectively.\footnote{Technically $\varepsilon$ and $\varepsilon^{-1}$ are only defined up to sign.} By \eqref{eq:pli} we see that the two isotropic lines in $V$ fixed by $\gamma$ are spanned by $u:=(\begin{smallmatrix}
        -a & a^2 \\ -1 & a
    \end{smallmatrix})$ and $\tilde u := (\begin{smallmatrix}
        -\tilde a & \tilde a^2 \\ -1 & \tilde a
    \end{smallmatrix})$, and an easy computation shows that the corresponding eigenvalues of $\gamma$, when acting on $V$, are $\lambda = \varepsilon^2$ and $\lambda^{-1}=\varepsilon^{-2}$. In particular, $F'=\mathbb{Q}$ and $E'$ is the real quadratic number field generated by any one of $a$, $\tilde a$, $\varepsilon$, or $\lambda$.

\end{example}



\begin{example}
    let $K$ be an imaginary quadratic field and $\mathcal{O}_K$ its ring of integers. let $\mathbf{V}$ be the quadratic space over $F=\mathbb{Q}$ with underlying space and quadratic form 
    \[
        \{X\in M_2(K): X+\overline{X}^a=0\}, \hspace{1cm} Q(X) = -\det X
    \]
    ($x^a$ denotes the adjugate of $x$). The signature is then $(3,1)$. The group $\Res_{K / \mathbb{Q}}\PSL_2$ acts by isometries on $\mathbf{V}$ by $g\cdot X = gX \overline{g}^{ - 1}$ and this gives an embedding $\Res_{K / \mathbb{Q}}\PSL_2\subset \mathbf{G}$. We have $\PSL_2(\mathbb{C})=G^0$, and the Bianchi group $\Gamma:=\PSL_2(\mathcal{O}_K)$. One can also give an isometry between the Poincar\'e upper half-space and $\mathcal{H}$ by
    \[
        \{z =x + yj\in \mathbb{H}: x\in \mathbb{C}, y\in \mathbb{R}_{ > 0}\} \cong \mathcal{H}, \hspace{1cm} z\mapsto  \frac{1}{y}\begin{pmatrix} - x & |z|^2 \\ - 1 & x\end{pmatrix},
    \]
    and this intertwines the action of $\PSL_2(\mathbb{C})$ by quaternionic linear fractional transformations and that of $G^0$. We have an isomorphism of varieties
    \begin{equation}\label{eq:rpli}
        \Res_{K /\mathbb{Q}}\mathbb{P}^1\cong\Iso_\mathbf{V},\hspace{1cm} (\alpha:\beta)\mapsto \Span \begin{pmatrix} - \alpha \overline{\beta} & |\alpha|^2 \\ - |\beta|^2 & \overline{\alpha}\beta\end{pmatrix} .
    \end{equation}
    which is equivariant for the action of $\Res_{K/\mathbb{Q}}\PSL_2$ on the left and $\mathbf{G}$ on the right.

    Now let $\gamma\in \Gamma:=\PSL_2(\mathcal{O}_K)$, $t=\tr\gamma$, and assume that $|t|>2$. 
    Let $\alpha,\tilde \alpha\in \mathbb{C}$ be the two fixed points of $\gamma$ and $\varepsilon$, $\varepsilon^{-1}$ the corresponding eigenvalues for $(\alpha,1)$ and $(\alpha',1)$. Then $\alpha$, $\tilde \alpha$, and $\varepsilon$ belong to the quadratic extension $L=K(\sqrt{\delta})$ of $K$, where $\delta=t^2-4$. By \eqref{eq:rpli} the isotropic lines in $V$ fixed by $\gamma$ are then spanned by
    \[
        u: =\begin{pmatrix}
        -\alpha & |\alpha|^2 \\ -1 & \overline{\alpha}
    \end{pmatrix}, \hspace{.5cm} \tilde u : =\begin{pmatrix}
        -\tilde \alpha & |\tilde \alpha|^2 \\ -1 & \overline{\tilde \alpha}
    \end{pmatrix},
    \]
    and it is not hard to compute that the corresponding eigenvalues of $\gamma$, when acting on $V$, are $\lambda=|\varepsilon|^2$ and $\lambda^{-1}=|\varepsilon|^{-2}$. Moreover, a direct computation gives $|\varepsilon|^2+|\varepsilon|^{-2}=(|t|^2+|\delta|) / 2$ so $F' = \mathbb{Q}(|\delta|)$.

    Thus $F'=\mathbb{Q}$ if and only if $|\delta|\in \mathbb{Q}$. If that is the case, then $\overline{L}=L$ and so $L$ is a biquadratic extension with imaginary quadratic subfields $K$ and $\mathbb{Q}(\sqrt{\delta})$ and real quadratic subfield $E'=\mathbb{Q}(|\epsilon|^2)$. By another similar calculation  $\varepsilon / \overline{\varepsilon}+\overline{\varepsilon} / \varepsilon = (|t|^2-|\delta|) /2$, so $\varepsilon / \overline{\varepsilon}$ is a $4$th or $6$th root of unity, as those are the only numbers of absolute value $1$ belonging to quadratic number fields. 

    If $|\delta|\not\in \mathbb{Q}$, then $E'=\mathbb{Q}(|\varepsilon|^2)$ and $\mathbb{Q}(\varepsilon / \overline{\varepsilon})$ are (Galois-conjugate) quartic ATR number fields. Their common Galois closure is $M:=\mathbb{Q}(|\varepsilon|^2,\varepsilon / \overline{\varepsilon}) = E' K = L\overline{L}$ which is a dihedral number field of degree $8$.

\begin{figure}[h]
    \centering
\begin{tikzpicture}[node distance = .6cm, auto]
\node (Q) {$\mathbb{Q}=F=F'$};
\node (A) [above of=Q, above of=Q, left of=Q, left of=Q, left of=Q] {$\mathbb{Q}(\sqrt{\delta})$};
\node (E) [above of=Q, above of=Q] {$E'$};
\node (K) [above of=Q, above of=Q, right of=Q, right of=Q, right of =Q] {$K$};
\node (L) [above of=Q, above of=Q, above of=Q, above of=Q] {$L$};
\draw[-] (Q) to node {} (A);
\draw[-] (Q) to node [swap] {} (E);
\draw[-] (Q) to node [swap] {} (K);
\draw[-] (A) to node  {} (L);
\draw[-] (E) to node  {} (L);
\draw[-] (K) to node  {} (L);
\end{tikzpicture}
\qquad\qquad
\centering
\begin{tikzpicture}[node distance = .6cm, auto]
\node (Q) {$\mathbb{Q}=F$};
\node (K) [above of=Q, above of=Q, left of=Q, left of=Q, left of=Q, left of=Q] {$\ast$};
\node (F1) [above of=Q, above of=Q] {$F'$};
\node (F2) [above of=Q, above of=Q, right of=Q, right of=Q, right of=Q, right of=Q] {$K$};
\node (F) [above of=K, above of=K] {$F'K$};
\node (E1) [above of=F1, above of=F1, left of=F1] {$E'$};
\node (E2) [above of=F1, above of=F1, right of=F1] {$\mathbb{Q}(\varepsilon / \overline\varepsilon)$};
\node (E3) [above of=F2, above of=F2, left of=F2] {$L$};
\node (E4) [above of=F2, above of=F2, right of=F2] {$\overline{L}$};
\node (L) [above of=Q, node distance = 3.6cm] {$M$};
\draw[-] (Q) to node {} (K);
\draw[-] (Q) to node [swap] {} (F1);
\draw[-] (Q) to node [swap] {} (F2);
\draw[-] (K) to node  {} (F);
\draw[-] (F1) to node  {} (F);
\draw[loosely dashed] (F2) to node  {} (F);
\draw[-] (F1) to node  {} (E1);
\draw[-] (F1) to node  {} (E2);
\draw[-] (F2) to node  {} (E3);
\draw[-] (F2) to node  {} (E4);
\draw[-] (F) to node  {} (L);
\draw[-] (E1) to node  {} (L);
\draw[-] (E2) to node  {} (L);
\draw[-] (E3) to node  {} (L);
\draw[-] (E4) to node  {} (L);
\end{tikzpicture}
\caption{Real quadratic case $|\delta|\in \mathbb{Q}$ to the left, and ATR case $|\delta|\notin \mathbb{Q}$ to right.}
\end{figure}
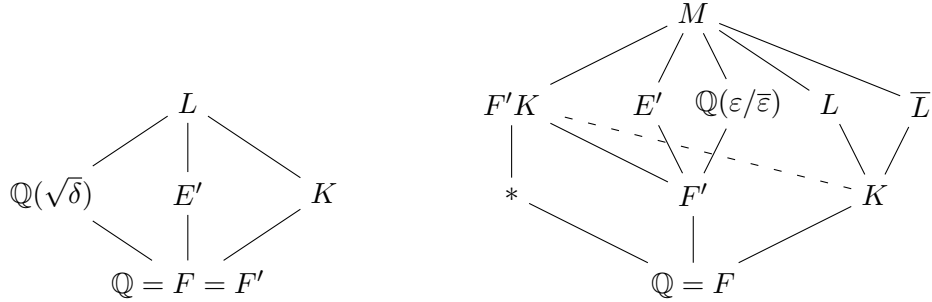

Note that as the quartic fields $E'$, $\mathbb{Q}(\varepsilon / \overline{\varepsilon})$, $L$, and $\overline{L}$ are Galois conjugates inside $M$, they determine each other up to reordering. In particular, $K=L\cap \overline{L}$ can be determined from $F'$ and $E'$, which gives an answer to \cite[Remark 4.1, 2.]{MR4863919}.
\end{example}

\begin{remark}
    If $\Gamma'$ contains torsion, then it may have multiple primitive generators which are not conjugate.
\end{remark}

\section{Special forms} \label{sec:2}

\subsection{Weighted cycles}

Following Kudla \cite{MR1427845} we define weighted cycles.

For $X\in V$, let $u(X):\mathcal{D}\to \mathbb{R}$ denote the function $z\mapsto (X,z)$. For a tuple of vectors $X=(X_1,\ldots,X_m)\in V^m$ we define $u(X):\mathcal{D}\to \mathbb{R}^m$ and $r(X):\mathcal{D}\to \mathbb{R}_{\geq0}$ by
\begin{align*}
    u(X) &=(u(X_1),\ldots,u(X_m)) ,
    \\
    r(X) &=\left(u(X_1)^2 +\cdots + u(X_m)^2\right)^{1/2}.
\end{align*}
We also let 
\[
    Q(X) := \frac{1}{2} \begin{pmatrix} (X_1,X_1) & \cdots & (X_1,X_m) 
    \\
\vdots & & \vdots
\\
(X_m,X_1) & \cdots & (X_m,X_m) \end{pmatrix} 
\in\Sym_m(\mathbb{R}).
\] 
We let $\mathcal{D}_X\subset \mathcal{D}$ and $c_X=\mathcal{D}_X\cap \mathcal{H}$ denote the zero locus of $u(X)$, or equivalently the zero locus of $r(X)$. We have $\mathcal{D}_X=\mathcal{D}_{\Span X}$ and $c_X = c_{\Span X}$ when $\Span X$ is positive semidefinite, and otherwise $\mathcal{D}_X$ and $c_X$ are empty. We let $\Gamma_X\subset \Gamma$ denote the stabilizer of $X$.

Now suppose that $X\in \mathbf{V}(F)^m$, and that $Q(X)$ is positive definite. Let
\[
    c(X)_\Gamma : \Gamma_X \backslash c_X \to M_\Gamma
\]
denote the associated special geodesic cycle. Let $T\in\Sym_m(F)$ be positive definite, and let $\varphi_f\in \mathscr{S}(\mathbf{V}(\mathbb{A}_{f})^m)^\Gamma$. We define the associated weighted geodesic cycle to be the formal linear combination
\[
    c(T,\varphi_f)_\Gamma := \sum_{\substack{X\in \mathbf{V}(F)^m / \Gamma \\ Q(X) = T}} \varphi_f(X) c(X)_\Gamma.
\]
When $Q(X)$ and $T$ are not positive definite we simply let $c(X)_\Gamma$ and $c(T,\varphi_f)_\Gamma$ be the empty cycle. We very often abuse notation and just write $c(X)$ and $c(T,\varphi_f)$, when $\Gamma$ is clear from context. For $a\in \GL_m^+(\mathbb{R})$ and $Y=Xa$ we also write $c(Y)=c(X)$, so that $c(Y)$ is defined whenever $\Span Y \subset V$ has a rational basis.

\subsection{Special forms}

We define forms $\varphi^0(X)\in \Omega^m(\mathcal{D})$ and $\psi^0(X)\in \Omega^{m-1}(\mathcal{D})$, for $X=(X_1,\ldots,X_m)\in V^m$, by
\begin{align*}
    \varphi^0(X) &=  du(X_1) \wedge \cdots \wedge du(X_m) e^{ - \pi r(X)^2}, \\
    \psi^0(X) &= \frac{1}{2}\sum_{i = 1}^m ( - 1)^{i - 1} du(X_1)\wedge\cdots \wedge \widehat{d}u(X_i)\cdots\wedge du(X_m) e^{ - \pi r(X)^2}.
\end{align*}
The following lemma is proved in exactly the same way as \cite{christensen2024linking}

\begin{lemma}\label{lem:bpl} Let $X\in V^m$. Then
    \begin{enumerate}[label = (\alph*),font=\upshape]
        \item for every $g\in \OO(V)$, we have $g^\ast \varphi^0(X)= \varphi^0(g^{-1}X)$ and $g^\ast\psi^0(X) = \psi^0(g^{-1}X)$;
        \item for every $h\in \SO(m)$, we have $\varphi^0(X h) = \varphi^0(X)$ and $\psi^0(Xh) = \psi^0(X)$;
        \item if $X_1$,\ldots, and $X_m$ are linearly dependent then $\varphi^0(X)=\psi^0(X)=0$;
        \item $\varphi^0(X)$ is closed;
        \item for $t>0$
        \begin{equation}\label{fte}
            \frac{d }{dt} \varphi^0(Xt^{1 /2}) = \frac{1}{t}d\psi^0(Xt^{1 /2}).
        \end{equation}
    \end{enumerate}
\end{lemma}

We shall also need the compatibility of these forms with restriction to totally geodesic submanifolds. To set this up, let $U\subset V$ be a positive definite subspace and denote $V'=U^\perp$ and $\mathcal{D}'=\mathcal{D}_U$. For $X'\in (V')^m$ let $\varphi'^{0}(X')\in \Omega^m(\mathcal{D}')$, and $\psi'^{0}(X')\in \Omega^{m-1}(\mathcal{D}')$ denote the forms defined by the same formulas as above.

\begin{lemma}\label{lem:fos}
    Let $X\in V^m$ and let ${X}\in {V}^m$ be the orthogonal projection of $X$. Then 
    \[
        \varphi^0(X)|_{\mathcal{D}'} = \varphi'^0(X')\hspace{1cm}\text{and}\hspace{1cm} 
        \psi^0(X)|_{{\mathcal{D}'}} = \psi'^0(X')
    \]
\end{lemma}
\begin{proof}
    For $ X'\in  V'$, let $u'(X'):{\mathcal{D}'}\to \mathbb{R}$ denote the function $z\mapsto ( X',z)$. Now if $X'$ is the projection of $X\in V$ and $z\in {\mathcal{D}'}\in  V$, then $(X,z) = (X',z)$, and hence $u(X)|_{X'} = u'(X')$. This implies the desired relations.
\end{proof}

We also introduce a form with singularities by
\[
    \widetilde\psi^0(X) : = -\int_1^\infty \psi^0(Xt^{\frac{1}{2}}) \frac{dt}{t}.
\]
This defines a smooth form away from $\mathcal{D}_X$. By Lemma \ref{lem:bpl}(c), $\widetilde\psi^0(X)=0$ when $X_1,\ldots,X_m$ are linearly dependent. 

Evidently, the three forms above are the pullback by $u(X)$ of universal forms $\varphi^0$, $\psi^0$, and $\widetilde\psi^0$ on $\mathbb{R}^m$. This will be useful in the next subsection. More explicitly, let $u_1,\ldots,u_m$ denote the coordinates on $\mathbb{R}^m$, and let $r$ denote the norm. Then $\varphi^0=e^{-\pi r^2} \nu$ and $\psi^0=e^{-\pi r^2}\rho$, where $\nu=du_1\wedge\cdots\wedge du_m$ is the volume form on $\mathbb{R}^m$ and $\rho$ is the contraction of $\nu$ by the radial vector field $r \frac{\partial}{\partial r} = \sum_{i=1}^m u_i \frac{\partial}{\partial u_i}$. We also have $\rho=r^m \omega$, where $\omega$ is the standard volume form on $S^{m-1}$, and so away from the origin
\begin{equation}\label{eq:utf}
    \widetilde\psi^0 = -\frac{1}{2}\int_1^\infty t^{\frac{m}{2}} e^{ - \pi r^2 t}\frac{dt}{t} \rho = - \frac{1}{2}\pi^{ - \frac{m}{2}} \Gamma\left(\frac{m}{2},\pi r^2\right) \omega,
\end{equation}
where $\Gamma(s,x)=\int_x^\infty t^s e^{-t}\frac{dt}{t}$ is the incomplete Gamma function. The transgression equation reads 
\begin{equation}\label{eq:ute}
    \frac{\partial}{\partial t} (t^{m/2} e^{-\pi r^2 t} \nu)
    =
    d(t^{\frac{m}{2}-1} e^{-\pi r^2 t} \rho).
\end{equation}

\begin{example}
    For $m=1$, we have $\nu=du$, $\rho = u$, and $\omega=\sgn(u)$. Hence
    $\varphi^0= e^{-\pi u^2}du$, $\psi^0 = ue^{-\pi u^2}$, and $\widetilde\psi^0 = -\frac{\sgn(u)}{2}\erfc(\sqrt{\pi}|u|)$, where $\erfc$ is the complementary error function. This matches with known formulas \cite{MR3796415}.
\end{example}

\subsection{Current relation}

The main property of $\widetilde\psi^0(X)$ is the following current relation.

\begin{proposition}\label{prp:tce}
    Away from $\mathcal{D}_X$ we have $d\widetilde{\psi}^0(X)=\varphi^0(X)$. Further, $\widetilde{\psi}^0(X)$ is locally integrable and satisfies the current equation 
    \[
        ( - 1)^{n - m + 1}\int_\mathcal{D} d\eta\wedge \widetilde\psi^0(X) = \int_\mathcal{D} \eta\wedge \varphi^0(X) - \int_{\mathcal{D}_X} \eta
    \]
    for any compactly supported smooth form $\eta\in \Omega_c^{n-m}(\mathcal{D})$.
\end{proposition}
\begin{proof}
This essentially follows from the corresponding analogous result for the universal forms $\varphi^0$, $\psi^0$, and $\widetilde \psi^0$. Away from the origin of $\mathbb{R}^m$ we can swap the exterior deivative with the integral in the definition of $\widetilde \psi^0$ and hence \eqref{eq:ute} implies $d\widetilde\psi^0 = \varphi^0$. Pulling back by $u(X)$ this implies that $d\widetilde\psi^0(X) = \varphi^0(X)$ away from $\mathcal{D}_X$.

Now whenever $\mathcal{D}_X$ is empty the transgression equation just follows by Stokes theorem. When $X_1$,\ldots, $X_m$ are linearly dependent there is also nothing to prove by Lemma \ref{lem:bpl}(c). The remaining case is when $Q(X)$ is positive definite. Here we emply the following lemma, which we will prove after finishing the current proof. 

\begin{lemma}\label{lem:tfb}
    If $Q(X)$ is positive definite then $\mathcal{D}_X\to \mathcal{D}\xrightarrow{u(X)} \mathbb{R}^m$ is a trivial fiber bundle.
\end{lemma}

We only need the fact that $u(X)$ is a fiber bundle, as this allows us to use the projection formula for integration along the fibers. The local integrability of $\widetilde\psi^0(X)$ now follows from \eqref{eq:utf} and the integrability of $\omega$ on $S^{m-1}$. Letting $f=u(X)_\ast \eta\in C_c^\infty(\mathbb{R}^m)$ we then have to show that
\[
    -\int_{\mathbb{R}^m} df \wedge \widetilde\psi^0 = \int_{\mathbb{R}^m} f\varphi^0 - f(0).
\]
Again we use the transgression equation, in the form 
\[
    \varphi^0 - t^{\frac{m}{2}}e^{ - \pi r^2t} \nu = -d\left(\int_1^t (t')^{\frac{m}{2}}e^{ - \pi r^2 t'} \frac{dt'}{t'} \rho\right).
\]
Integrating against $f$ and applying Stokes theorem we get 
\[
    \int_{\mathbb{R}^m} f\varphi^0 - \int_{\mathbb{R}^m}ft^{\frac{m}{2}} e^{ - \pi r^2t} \nu = \int_{\mathbb{R}^m}df\wedge \int_1^t (t')^{\frac{m}{2}}e^{ - \pi r^2 t'} \frac{dt'}{t'} \rho.
\]
Finally we take $t\to\infty$. The second term converges to $f(0)$ from the classic fact that $t^{m/2}e^{-\pi r^2t}$ converges towards the delta distribution at $0$. For the right-hand side we have
\[
    \int_{\mathbb{R}^m}df\wedge \int_t^\infty (t')^{\frac{m}{2}}e^{ - \pi r^2 t'} \frac{dt'}{t'} \rho = \int_{\mathbb{R}^m}df\wedge\int_{r^2t}^\infty (t')^{\frac{m}{2}}e^{ - \pi t'} \frac{dt'}{t'} \omega
\]
which converges to $0$ by dominated convergence.
\end{proof}

\begin{proof}[Proof of Lemma \ref{lem:tfb}]
Let $U=\Span X$ and let $x\mapsto x_U$ denote the orthogonal projection to $U$. The map 
\[
    \mathcal{D}\to U\times \mathcal{D}_U,\hspace{1cm} z \mapsto \left(z_U,\frac{z - z_U}{\sqrt{ -Q(z - z_U)}}\right)
\]
is a diffeomorphism since it has a smooth inverse $(x,y)\mapsto x+y\sqrt{Q(x)+1}$. The map $(X,-):U\to \mathbb{R}^m$ is also a diffeomorphism, hence $\mathcal{D}\cong \mathbb{R}^m\times \mathcal{D}_U$. Clearly this identifies $\mathcal{D}_U$ with $0\times \mathcal{D}_U$, and the associated map $\mathcal{D}\to \mathbb{R}^m$ is just $u(X)$ since $(X,z_U)=(X,z)$.

\end{proof}

\subsection{Averaging} 

Let $Y\in V^m$ be a vector such that $\Span Y\subset V$ has a $F$-rational basis, and let $\phi^0$ be either $\varphi^0$, $\psi^0$ or $\tilde\psi^0$. We let 
\begin{equation}\label{eq:aXf}
    \alpha(Y,\phi^0)_\Gamma : = \sum_{\gamma\in \Gamma_Y\backslash \Gamma} \phi^0(Y).
\end{equation}
We will often abuse notation and write simply $\alpha(Y,\phi^0)$. For $T=\Sym_m(F)$, $v\in\Sym_m(\mathbb{R})_{+}$, and $\varphi_f\in \mathscr{S}(\mathbf{V}(\mathbb{A}_{f})^m)$ we also define the weighted average
\begin{equation}\label{eq:TTf}
    \Theta(T,\varphi_f,\phi^0,v) := \sum_{\substack{X\in \mathbf{V}(F)^m \\ Q(X) = T}} \varphi_f(X) \phi^0(Xa)
\end{equation}
where $a\in\GL_m^+(\mathbb{R})$ is any matrix with $v=\tp{a} a$ (this does not depend on the choice of $a$ by Lemma \ref{lem:bpl}(b)).

For $\phi^0=\varphi^0$ and $\phi^0 = \psi^0$ the series \eqref{eq:aXf} and \eqref{eq:TTf} converge normally in $\Omega^\ast(\mathcal{D})$, and hence they define smooth $\Gamma$-invariant forms on $\mathcal{D}$ and consequently also smooth forms on $M$. This is a consequence of the following standard lemma. We introduce the following notation
\[
    \phi_\infty(X_\infty) := \exp(-2\pi \sum_{i = 1}^r\tr Q(X_\infty^{(i)}))\phi^0(X_\infty^{(1)}) \hspace{1cm}\text{for}\; X_\infty\in \mathbf{V}(\mathbb{A}_\infty)^m.
\]
Here $X_\infty= (X_\infty^{(1)},\ldots, X_\infty^{(r)})$.
\begin{lemma}\label{lem:cot}
    Let $\phi^0$ be either $\varphi^0$ or $\psi^0$. For any $a_\infty\in \GL_m^+(\mathbb{A}_\infty)$ the series
    \[
        \sum_{X\in \mathbf{V}(F)^m} \varphi_f(X) \phi_\infty(Xa_\infty)
    \]
    converges normally in $\Omega^\ast(\mathcal{D})$. In particular \eqref{eq:aXf} and \eqref{eq:TTf} converge normally in $\Omega^\ast(\mathcal{D})$ for any $Y$ and $T$.
\end{lemma}
\begin{proof}
This is very similar to \cite[Lemma 2.7]{christensen2024linking} and \cite{MR3783423}. We write some of it out as we will use the notation later. The Siegel majorant of $Q$ at $z\in \mathcal{D}$ is given by 
\[
    (X,Y)_z = (X,Y) + u(X,z)u(Y,z),\hspace{1cm} X,Y\in V
\]
and $Q_z(X) = \frac{1}{2}(X,X)_z$. For $X\in V^m$ we let $Q_z(X) = \frac{1}{2}((X_i,X_j)_z)_{1\leq i,j\leq m}$. Then for $X_\infty\in \mathbf{V}(\mathbb{A}_\infty)^m$ we got
\begin{equation}
    \tr Q_z(X_\infty^{(1)}) + \sum_{i = 2}^r \tr Q(X_\infty^{(i)}) = \sum_{i = 1}^r\tr Q(X_\infty^{(i)}) + \frac{1}{2}r(X_\infty^{(1)},z),
\end{equation}
and hence
\[
    \phi_\infty(X_\infty,z) = \exp(-2\pi \left(\tr Q_z(X_\infty^{(1)})+\sum_{i=2}^r\tr Q(X_\infty^{(i)})\right))p(X_\infty^{(1)},z),
\]
where $p(X,z)$ is a homogeneous polynomial of degree $m$ in $X\in V^m$ and a smooth form in $z\in \mathcal{D}$. This follows easily from the formulas for $\varphi^0$ and $\psi^0$, and that $u(X,z)$ is linear in $X$. Then for any compact region $K$ one can find a positive definite quadratic form $q_+^{(1)}:V^m\to \mathbb{R}_{\geq0}$ with $q_+^{(1)}(X)\leq \tr Q_z(Xa^{(1)})$ for $z\in K$ and $X\in V^m$, and finally to conclude one can compare with theta series for the positive definite quadratic form
\[
    q_+(X_\infty) := q_+ ^ {(1)}(X_\infty^{(1)}) +\sum_{i=2}^r\tr Q(X_\infty^{(i)}a^{(i)})
\]
on $\mathbf{V}(\mathbb{A}_\infty)^m$.

Note that if $X\in \mathbf{V}(F)^m$, $Q(X)=T$, and $a\in \GL_m^+(\mathbb{R})$, then
\[
    \phi^0(X^{(1)}a) = \exp( 2\pi \left(\tr (\tp{a}T^{(1)}a) +\sum_{i = 2}^r\tr T^{(i)}\right))\phi_\infty((X^{(1)}a,X^{(2)},\ldots,X^{(r)})) .
\]
Hence the convergence of \eqref{eq:aXf} and \eqref{eq:TTf} follows.
\end{proof}

For $\phi^0=\widetilde\psi^0$, the situation is a little more complicated due to the singularities of the various associated forms.

\begin{lemma} \label{lem:tcl} For $\phi^0=\widetilde\psi^0$ the series \eqref{eq:aXf}, resp. \eqref{eq:TTf}, converges normally on any compact subset of $\mathcal{D} - \bigcup_{\gamma\in\Gamma_Y\backslash \Gamma} c_{\gamma^{-1}Y}$, resp.  $\mathcal{D} - \bigcup_{\substack{X\in \mathbf{V}(F)^m \\ Q(X)=T}} c_{X}$. In particular, they define smooth forms on $M$ away from the support of $c(Y)_\Gamma$, resp. $c(T,\varphi_f)_\Gamma$.
\end{lemma}
\begin{proof}
We prove the statement for \eqref{eq:TTf}, since the other follows similarly. 
For convenience denote $\mathcal{S} = \{X\in\mathbf{V}(F)^m:Q(X)=T\}$. Let $K\subset\mathcal{D}-\bigcup_{}$ be a compact subset, and let $q^{(1)}_+:V^m\to \mathbb{R}$ be a positive definite quadratic form with $q^{(1)}_+(X)\leq \sum \tr Q_{z}(Xa)$ for all $z\in K$ and $X\in V^m$. Then
\begin{equation}\label{eq:rpb}
    r(Xa,z)^2 \geq 2 \left(q_ +^{(1)} (X) -\tr Q(Xa) \right), \hspace{1cm} X\in V^m,\; z\in K.
\end{equation}
Let $q_+:\mathbf{V}(\mathbb{A}_\infty)^m\to \mathbb{R}$ be the positive definite quadratic form given by the sum of $q_+^{(1)}$ and the completions of $Q$ at $\sigma_2,\ldots,\sigma_r$. Then \eqref{eq:rpb} implies that
\[
    r(Xa,n(x))^2 \geq 2 \left(q_ +(X) - \tr(\tp{a}T^{(1)}a) -\sum_{i = 2}^r\tr T^{(i)}\right), \hspace{1cm} X \in \mathcal{S},\; z\in K.
\]
One can find a lattice $ \Lambda\subset \mathbf{V}(F)^m $ that contains all vectors $X$ with $\varphi_f(X)\neq0$, and hence also a lattice in $\mathbf{V}(\mathbb{A}_\infty)^m$ containing the subset
\[
    \{(X^{(1)}a,X^{(2)},\ldots,X^{(r)}):X\in \mathcal{S}, \varphi_f(X)\neq 0\} \subset \mathbf{V}(\mathbb{A}_\infty)^m.
\]
There are only finitely many lattice vectors with a given bounded value of $q_+$, so as an upshot we obtain a constant $R=R_{T,a}>0$ such that for all $X\in \mathcal{S}$ with $\varphi_f(X)\neq0$ and $z\in K$
\begin{equation}\label{eq:rmin}
    r(Xa,z)\geq \max\{\sqrt{q_+(X)},R\}.
\end{equation}
Now let $p(X,z)$ be the function from the proof of Lemma \ref{lem:cot} for $\phi^0=\psi^0$. Then for $Y\in V^m$
\[
    \widetilde\psi^0(Y,z) = \frac{1}{2}\pi^{\frac{m}{2}}\Gamma\left(\frac{m}{2},\pi r(Y,z)^2\right)\frac{p(Y,z)}{r(Y,z)^m}.
\]
Note that if $n>0$, $\sigma>0$, and $x>x_0$ for a fixed $x_0>0$, then $\left|\frac{d^n}{dx^n} \Gamma(\sigma,x) \right| < C_n x^{\sigma-1}e^{-x}$ for some constant $C_n>0$. Hence the result now follows by a comparison with a theta series.
\end{proof}

\begin{corollary}\label{cor:dcc}
    With notation as above:
    \begin{enumerate}[font = \upshape, label = (\alph*)]
    \item Away from the support of $c(X)_\Gamma$ we have $d\alpha(X,\widetilde\psi^0)_\Gamma = \alpha(X,\varphi^0)_\Gamma$. Further $\alpha(X,\widetilde\psi^0)_\Gamma$ is locally integrable and satisfies the current equation
        \[
            ( - 1)^{n - m + 1}\int_{M} d\eta\wedge \alpha(Xa,\widetilde \psi^0)_\Gamma = \int_M \eta\wedge \alpha(Xa,\varphi^0) - \int_{c(X)_\Gamma} \eta
        \]
        for any compactly supported smooth form $\eta\in \Omega^{n-m}_c(M)$.
        \item Away from the support of $c(T,\varphi_f)_\Gamma$ we have $d\Theta(T,\varphi_f,\widetilde\psi^0,v) = \Theta(T,\varphi_f,\varphi^0,v)$. Further $\Theta(T,\varphi_f,\widetilde\psi^0,v)$ is locally integrable and satisfies the current equation
        \[
            ( - 1)^{n - m + 1}\int_{M} d\eta\wedge \Theta(T,\varphi_f,\widetilde\psi^0,v)_\Gamma = \int_M \eta\wedge \Theta(T,\varphi_f,\varphi^0,v)_\Gamma - \int_{c(T,\varphi_f)_\Gamma} \eta
        \]
        for any compactly supported smooth form $\eta\in \Omega^{n-m}_c(M)$.
        \end{enumerate}
\end{corollary}

\begin{proof}
    This follows by Proposition \ref{prp:tce} and Lemma \ref{lem:tcl} together with standard partition of unity argument.
\end{proof}

Corollary \ref{cor:dcc} implies that the currents defined by $\alpha(Y,\widetilde{\psi}^0)$ and $\Theta(T,\varphi_f,\widetilde{\psi}^0,v)$ are sparks in the sense of \cite{MR1993742}. They thus define differential characters on $M$ which we will denote by $\hat c(Y)$ and $\hat{c}(T,\varphi_f,v)$. These give a refinement of the homology classes of the special cycles $c(Y)$ and $c(T,\varphi_f)$.

\section{Linking numbers}\label{sec:3}

To define linking numbers in general it is necessary to keep track of additional homological data. If $a:A\to X$ is a smooth map of manifolds we define the relative homology group $H_m(a)=H_\ast(X,a)$ to be the $m$th homology of the mapping cone of the chain map $a_\ast:C_\bullet(A)\to C_\bullet(X)$. Here we take $C_\bullet$ to be the complex of smooth chains to allow for integration. Thus an element of $H_m(a)$ is represented by a pair $(\sigma,\mu)\in C_m(X)\times C_{m-1}(A)$ with $\partial \sigma = a_\ast \mu$, and such an element is zero if there exists $(\sigma',\mu')\in C_{m+1}(X)\times C_m(A)$ with $\sigma = \partial \sigma'-a_\ast \mu'$ and $\mu = \partial \mu'$. If $A$ is a submanifold of $X$ and $a$ the inclusion, then the map $\sigma\mapsto (\sigma,0)$ gives an isomorphism between the usual groups $H_m(X,A)=H_m(C_\bullet(X) / C_\bullet(A))$ and $H_\ast(a)$.



We now return to the setup of the previous section. We assume from now on that $\Gamma $ is torsion-free. Let $T=\Sym_m(F)$, $v\in\Sym_m(\mathbb{R})_{+}$, and $\varphi_f\in \mathscr{S}(\mathbf{V}(\mathbb{A}_{f})^m)^\Gamma$. We consider a smooth map $c:C\to M$, where $C$ is an oriented closed manifold of dimension $m-1$. Suppose that $s=[(\sigma,\ast)]\in H_m(M,c)$ is class which maps to the fundamental class of $C$ in $H_m(M)$. Then we define
\begin{equation}\label{eq:idef}
    \iota(T,\varphi_f;s) := \lim_{t\to\infty} \int_\sigma \Theta(T,\varphi_f,\varphi^0,vt).
\end{equation}
This is independent of the choice of representative for $s$, since by Lemma \ref{lem:bpl}(d) the integral over any exact chain is zero and the pullback of $\Theta(T,\varphi_f,\varphi^0,vt)$ to $C$ is zero. 
We have also suppressed $v$ from the notation, in line with the following theorem.
\begin{theorem}\label{thm:ili}
    Assume that $c$ is a union of special geodesic cycles of general type. Then 
    \begin{enumerate}[label = (\roman*),font=\upshape]
        \item the limit in \eqref{eq:idef} exists for all $T$, $\varphi_f$, and $v$, and it is independent of $v$;
        \item if $T^{(1)}$ is not positive definite then $\iota(T,\varphi_f;s)=0$.
    \end{enumerate}
\end{theorem}
The proof of Theorem \ref{thm:ili} will occupy the remainder of this section. 

\begin{remark}
    The proof of Theorem \ref{thm:ili} will show that $\iota(T,\varphi_f;s)$ is an integer if $c$ and $c(T,\varphi_f)$ have no intersections. In general we only confirm that $\iota(T,\varphi_f;s)$ is rational in general. It is conceivable that $\iota(T,\varphi_f;s)$ is always an half-integer, as in \cite{christensen2024linking}, but our proof does not verify this.
\end{remark}

\begin{remark}
    If $H_m(M)$ is torsion then $\iota(T,\varphi_f;s)$ also does not depend on the class $s$. 
\end{remark}
\begin{remark}
    If $S$ is an oriented compact manifold with boundary $\partial S=C$ and $s:S\to M$ is a smooth map with $\partial s = c$, then the pushforward of the fundamental class of $S$ determines a class of $H_m(M,c)$ that we suggestively denote $s$ as well. Then
    \[
        \iota(T,\varphi_f;s) = \lim_{t\to\infty}\int_S s^\ast \Theta(T,\varphi_f,\varphi^0,vt)
    \]
    This was the approach we took to define $\iota(T,\varphi_f;s)$ in \cite{christensen2024linking}. However, if $m>2$ then actually it seems to be an open question whether or not $C$ is always the boundary of an oriented compact manifold, regardless of its immersion in $M$ \cite[Question 1.4.]{chen2025non}. 
\end{remark}

\subsection{Vanishing contributions}

We first prove a vanishing statement, which both implies (iii) and will also be used several times in the proof of (i) and (ii). 

For any subset 
\[
    \mathcal{S}\subset \{X\in \mathbf{V}(F)^m:Q(X) = T \}
\]
we define the partial sum
\[
    \Theta(\mathcal{S},\varphi_f,\varphi^0,v) := \sum_{X\in S}\varphi_f(X)\varphi^0(Xa).
\]

\begin{proposition}\label{prp:psv}
    Let $N$ be a compact manifold of dimension $m$ (possibly with corners) and let $n:N\to \mathcal{H}$ be a smooth map whose image does not intersect $c_X$ for any $X\in \mathcal{S}$. Then 
    \[
        \lim_{t\to\infty}\int_N n^\ast \Theta(\mathcal{S},\varphi_f,\varphi^0,vt) = 0.
    \]
\end{proposition}
\begin{proof}
Fix a volume form $\nu$ on $N$. Then
\[
    n^\ast \varphi^0(X) = e^{ - \pi r(X,n(x))^2} f(X,x)\nu
\]
where $f$ is a homogeneous polynomial in $X\in V^m$ and smooth in $x\in N$. We obtain a bound of $|\int_N n^\ast \Theta(S,\varphi_f,\varphi^0,vt)|$ by 
\begin{equation}\label{eq:bps}
    \sum_{X\in S} |\varphi_f(X)| t^{\frac{m}{2}} e^{ - \pi r_0(Xa)^2 t}f_0(Xa)
\end{equation}
where
\[
    r_0(Y) : = \inf_{x\in N} r(Y,n(x)), \hspace{1cm} f_0(Y) = \sup_{x\in N}|f(Y,n(x))|, \hspace{1cm}\text{for}\;Y\in V^m.
\]
Because the image of $n$ does not meet $c_X$ for any $X\in \mathcal{S}$ we have $r_0(Xa)>0$ for every $X\in \mathcal{S}$. Therefore \eqref{eq:bps} converges termwise to $0$ as $t\to\infty$. We now bound the series by a positive definite theta function and apply dominated convergence.

Let $q^{(1)}_+:V^m\to \mathbb{R}$ be a positive definite quadratic form such that $q^{(1)}_+(X)\leq \sum \tr Q_{n(x)}(Xa)$ for all $x\in N$ and $X\in V^m$, and let $q_+:\mathbf{V}(\mathbb{A}_\infty)^m\to \mathbb{R}$ be the positive definite quadratic form given by the sum of $q_+^{(1)}$ and the completions of $Q$ at $\sigma_2,\ldots,\sigma_r$. In the same way as in the proof of Lemma \ref{lem:tcl}, we obtain a constant $R=R_{T,a}>0$ such that
\begin{equation}\label{eq:rmin2}
    r_0(Xa)\geq \max\{\sqrt{q_+(X)},R\}, \hspace{1cm} X\in \mathcal{S}, \varphi_f(X)\neq0.
\end{equation}
For $Y\in V^m$, the function $te^{-\pi r_0(Y)^2 t}$ is strictly decreasing for $t> 1/(\pi r_0(Y)^2)$, and we conclude from \eqref{eq:rmin2} that $te^{-\pi r_0(Y)^2 t}\leq t_0 e^{-\pi q_+(Y)t_0}$ for $t>t_0:= 1/(\pi R^2)$. As a result 
\[
    te^{-\pi r_0(Y)^2 t}f_0(Y) \leq C_0 e^{ - C_1 q_+ (Y)}, 
\]
for constants $C_0,C_1>0$. This gives the required estimate for applying dominated convergence.
\end{proof}

\begin{proof}[Proof of Theorem \ref{thm:ili}(iii)]
By a barycentric subdivision, we may assume that every chain of $\sigma$ can be lifted to $\mathcal{H}$, and so it suffices to see that the integral over such chains tends to $0$.

    If $T^{(1)}$ is not positive semidefinite then this follows from Proposition \ref{prp:psv} by taking $S= \{X\in V(F)^m:Q(X)=T\}$. If $T^{(1)}$ is positive semidefinite and nonsingular then $\Theta(T,\varphi_f,\varphi^0,v)=0$ by Lemma \ref{lem:bpl}(c)
\end{proof}

\subsection{Preliminary reductions} 

Following the last subsection, we assume that $T^{(1)}$ is positive definite. The sum in \eqref{eq:TTf} contains finitely many $\Gamma$-orbits \cite[p.132]{MR1079646} and we have
\begin{equation}\label{eq:fsd}
    \int_\sigma \Theta(T,\varphi_f,\varphi^0,vt) = \sum_{\substack{X\in \mathbf{V}(F)^m / \Gamma, \\ Q(X) = T}} \varphi_f(X) \int_\sigma \alpha(Xat^{\frac{1}{2}},\varphi^0).
\end{equation}
Hence it suffices to prove the analogous results for the limit
\begin{equation}\label{eq:aix}
    \lim_{t\to\infty} \int_\sigma \alpha(Xat^{\frac{1}{2}},\varphi^0).
\end{equation}

\begin{proposition}\label{prp:pdi}
    Suppose that $Q(X)$ is positive definite, and $c$ is a union of special geodesics of general type. If $s=[(\sigma,\mu)]\in H_m(M,c)$ has boundary $[C]\in H_{m-1}(C)$, then the limit in \eqref{eq:aix} exists and is rational.
\end{proposition}

We first note that it suffices to prove Proposition \ref{prp:pdi} under the assumption that $c(X)$ is an embedding. By the Jaffee lemma \cite[Proposition 2.2.]{MR653949} we may find a finite index normal subgroup $\Gamma'\subset \Gamma$ such that the map $\Gamma_X'\backslash c_X\to M_{\Gamma'}$ is an embedding. Now define $c':C'\to M_{\Gamma'}$ by the fibre product
\[
\begin{tikzcd}
    C' \ar[d] \ar[r,"c'"] & M_{\Gamma'} \ar[d]
    \\
    C \ar[r,"c"'] & M_\Gamma  
\end{tikzcd}
\]
By barycentric subdivision we may choose the representative $(\sigma,\gamma)$ such that every simplex of $\sigma$ and $\gamma$ can be lifted in $[\Gamma:\Gamma']$ ways to a simplex in $M_{\Gamma'}$. We let $\sigma'$ and $\gamma'$ denote the $\Gamma' / \Gamma$-equivariant chains obtained in this way. Then we have the projection formula
\begin{equation}\label{eq:ssd}
    \int_\sigma \alpha(Xt^{\frac{1}{2}},\varphi^0)_\Gamma = \frac{1}{[\Gamma_X:\Gamma_X\cap \Gamma']}\int_{\sigma'} \alpha(Xt^{\frac{1}{2}},\varphi^0)_{\Gamma'}.
\end{equation}
So it suffices to study the limits of the integral on the right side. Now the main idea, which we will realize in the next two subsections, is to perturb $c$ slightly so that it no longer intersects $c(X)$ and then apply the following simple lemma.
\begin{lemma}\label{lem:prs}
    Let $c:C\to M$ be a smooth map, with $C$ any compact manifold. Suppose $[(\sigma_i,\mu_i)]\in H_m(M,c)$, $i=1,2,3$, are such that $[\mu_3]=[\mu_1]+[\mu_2]$ in $H_{m-1}(C)$. If the limit in \eqref{eq:aix} exists for both $\sigma=\sigma_1$ and $\sigma=\sigma_2$, then it also exists for $\sigma=\sigma_3$.
\end{lemma}
\begin{proof}
We can find a chain $\sigma'$ in $C$ of degree $m$ such that $\partial \sigma' +\mu_3=\mu_1+\mu_2$. Since $\partial \sigma_i = c_\ast \mu_i$, we get that $\sigma'':=c_\ast \sigma' +\sigma_3-\sigma_1-\sigma_2$ is a cycle. But the limit in \eqref{eq:aix} exists for $\sigma''$ (since it is constant), and for $c_\ast\sigma'$ since $c^\ast \varphi^0(Xat^{1/2})=0$.
\end{proof}

\subsection{A simpler case}

We assume from now on that $c(X)$ embeds. In this subsection we prove

\begin{lemma}
    Let $c:C\to M$ be a smooth map, with $C$ any compact oriented manifold. Suppose that $c$ and $c(X)$ are disjoint and $s=[(\sigma,\mu)]\in H_m(M,c)$ has boundary $[C]\in H_{m-1}(C)$. Then the limit in \eqref{eq:aix} exists and belongs to $\mathbb{Z}$.
\end{lemma}
\begin{proof}
By barycentric subdivision we can assume that no simplex of $\sigma$ intersects both $c$ and $c(X)$, and that all simplices lift to $\mathcal{H}$. By Thom transversality \cite{guillemin2010differential}, we can further assume that $s$ is represented by $(\sigma,\mu)$ where all simplices of $\sigma$ that do not intersect $c$ are transverse to $c(X)$. For this representation suppose that $\sigma=\sum_{i} a_i \sigma_i$, $a_i\in\mathbb{Z}$. For all the $\sigma_i$ that intersect $c$, and hence do not intersect $c(X)$, we have 
\begin{equation}
    \lim_{t\to\infty} \int_{\sigma_i} \alpha(Xat^{\frac{1}{2}},\varphi^0) = 0 
\end{equation}
by Proposition \ref{prp:psv}. 

For $\sigma_i$ that do not intersect $c$, and hence is transverse to $c(X)$, we follows the argument in \cite{christensen2024linking}. Let $\widetilde{\sigma}_i:\Delta^m\to \mathcal{H}$ denote a lift of $\sigma_i$, where $\Delta^m\subset\mathbb{R}^m$ is the standard simplex. We got
\begin{equation}\label{eq:ubn}
    \int_{\sigma_i} \alpha(Xat^{\frac{1}{2}},\varphi^0) = \sum_{X_1\in \Gamma X} \int_{\widetilde\sigma_i} \varphi^{0}(X_1 at^{\frac{1}{2}})
\end{equation}
Suppose that $\widetilde{\sigma}_i$ is transverse to $c_{X_1}$. For any $x\in\widetilde{\sigma}_i^{-1}(c_{X_1})$, this implies that $u(X_1)\circ\widetilde{\sigma}_i$ is regular at $x$, and hence there exists an open neighborhood $x\in A_x\subset \Delta^m$ which is mapped diffeomorphically onto an open disc $0\in B_x\subset \mathbb{R}^m$. Let $\varepsilon_x$ be $+1$ or $-1$ corresponding to as $u(X)\circ \widetilde{\sigma}_i$ is orientation-preserving or not. Then 
\begin{equation}\label{eq:ecc}
    \lim_{t\to\infty}\int_{A_x} \widetilde{\sigma}_i^\ast \varphi^{0}(X_1 at^{\frac{1}{2}}) 
    = 
    \varepsilon_x \lim_{t\to\infty}\int_{B_x} t^{\frac{m}{2}}e^{ -\pi t (u_1^2 +\cdots u_m^2)} du_1\wedge \cdots \wedge du_m 
    =
    \varepsilon_x
\end{equation}
Putting \eqref{eq:ubn} and \eqref{eq:ecc} together with Proposition \ref{prp:psv} gives 
\[
    \lim_{t\to\infty} \int_{\sigma_i} \alpha(Xat^{\frac{1}{2}},\varphi^0) = \sum_{x} \varepsilon_x.
\]
\end{proof}

\subsection{General case}

Now consider a geodesic cycle of general type. We continue to assume that $c(X)$ is embedded, and our strategy will be to construct another map $c':C'\to M$ and a chain $s\in H_m(M,c\sqcup c')$, by slightly pertubing $c$, such that 
\begin{enumerate}[label =(\alph*), font=\upshape]
    \item $\partial s = [C]-[C']$;
    \item $c'$ is disjoint to $c(X)$;
    \item \eqref{eq:aix} equals $\pm\frac{1}{2}$.
\end{enumerate}
For a union of geodesic cycles, we may repeat this process for every component and so in conjunction with Lemma \ref{lem:prs} this will imply Proposition \ref{prp:pdi} and hence Theorem \ref{thm:ili}.

Now fix $(\mathbf{V}',\eta)$ and $\Gamma_0'\subset \Gamma'$ such that $c=c_0'$ and $C=\Gamma_0'\backslash c_U$, as in section \ref{sec:1}. We use an explicit tubular neighborhood construction based on lemma \ref{lem:tfb}. For a continuous function $\varepsilon:\Gamma_X\backslash c_X\to \mathbb{R}_{>0}$, consider 
\[
    T_{X,\varepsilon} =\{(z,u)\in \mathcal{H}:  |u|< \varepsilon(z)\}
\]
and the associated map 
\[
    t_\varepsilon(X): \Gamma_X\backslash T_{X,\varepsilon}\to M.
\]
We pick $\varepsilon$ such that $t_\varepsilon(X)$ is an embedding, and we abuse notation and denote by $t_\varepsilon(X)$ the associated tubular neighborhood as well. Actually by possibly picking $\varepsilon$ smaller, we can also assume that the closure of $t_\varepsilon(X)$ embeds and does not meet $c$. 

Let $C_\varepsilon\subset C$ denote the preimage of $t_\varepsilon(X)$.
For any $\gamma\in \Gamma$ the map
\begin{equation}\label{eq:ccc}
    \left(\gamma^{ - 1 }\Gamma_0' \gamma \cap \Gamma_X\right) \backslash \left(c_{\gamma^{ - 1}U} \cap T_{X,\varepsilon}\right) \to C_\varepsilon, \hspace{1cm} z \mapsto \Gamma_0'\gamma z
\end{equation}
is an open embedding since $t_\varepsilon(X)$ embeds and the action of $\Gamma$ is free. The same proof as \cite{christensen2024linking}(unfolding lemma for geodesics) shows that $C_\varepsilon$ is the disjoint union of the images of the embeddings \eqref{eq:ccc}, as $\gamma$ varies through a set of representatives of $\Gamma_0'\backslash \Gamma / \Gamma_X$.



Thus, if $\widetilde C_\varepsilon\subset C_\varepsilon$ is a connected component we have a uniformization
\[
    \widetilde C_\varepsilon \cong \left(\gamma^{ - 1 }\Gamma_0' \gamma \cap \Gamma_X\right) \backslash \left(c_{\gamma^{ - 1}U} \cap T_{X,\varepsilon}\right).
\]
for some $\gamma\in \Gamma$. Let us assume that $c$ intersects $c_X$ on $\widetilde C_\varepsilon$, and hence that $c_{\gamma^{-1}U}\cap c_X$ is non-empty, since otherwise there is no need to perturb. Thus $W=\gamma^{-1}U+\Span X$ is positive definite and $c_{\gamma^{-1}U}\cap c_X=c_W$. Since $\dim_{\mathbb{R}}W\leq n$, $\dim_\mathbb{R} U = n+1-m$, and $\dim_\mathbb{R} \Span X=m$, it follows that $\gamma^{-1}U\cap \Span X\neq0$. 

Notice that the image of $\widetilde C_\varepsilon$ lies inside the zero locus of $u(Y)$ for $Y\in \gamma^{-1}U\cap \Span X$.
We fix a basis $Y$ of $\Span X$ such that $Y_{\ell+1},\ldots,Y_m$ is a basis for $\gamma^{-1}U\cap \Span X$. Let $X=Yb$, for $b\in\GL_m(\mathbb{R})$, and denote $v= u(X)b^{-1}$. By Lemma \ref{lem:tfb}, we can fix a diffeomorphism $T_{X,\varepsilon}\cong c_X\times \mathbb{R}^m$ of the form $z\mapsto(\ast,v(z))$. Then the restriction of $c$ takes the form  
\[
    c|_{\widetilde C_\varepsilon}:\widetilde C_\varepsilon\to \Gamma_X\backslash T_{X,\varepsilon}, \hspace{1cm} x\mapsto (\widetilde z(x),v_1(x),\ldots v_{\ell}(x),0,\ldots,0) 
\]
for some map $\tilde z:\widetilde C_\varepsilon\to \Gamma_X\backslash c_X$. 

Take $\rho:[0,\infty)\to [0,1]$ to be a smooth function satisfying $\rho(t)=0$ for $t> \frac{2}{3}$ and $\rho(t)=1$ for $t< \frac{1}{3}$. We consider the map
\[
    \widetilde s: \widetilde C_\varepsilon\times [0,\delta]\to \Gamma_X\backslash T_{X,\varepsilon}, \hspace{.8cm} (x,y)\mapsto \left(\widetilde z(x),v_1(x),\ldots v_{\ell}(x),0,\ldots,0,y \rho\left(\frac{|v(x)|}{\varepsilon||b || }\right)\right).
\]
For sufficiently small $\delta$ the image of this map is inside $t_{\varepsilon}(X)$. We may repeat this construction for every connected component of $C_\varepsilon$ and obtain a map $s:C\times [0,\delta]\to M$.

Now we show that, for a given component, the limit in \eqref{eq:aix} converges to $\pm\frac{1}{2}$ if $\ell=m-1$ and zero else. It follows from Proposition \ref{prp:psv} that
\[
    \int_{\widetilde s}\sum_{X_1\in \Gamma X - \{X\}} \varphi^0(X_1 at^{\frac{1}{2}})\to 0
\]
as $t\to\infty$. 
If $\ell<m-1$, then the pullback of $\varphi^0(Xat{1/2})$ is zero, so it also follows in this case that 
\[
    \int_{\widetilde s}\varphi^0(X at^{\frac{1}{2}})\to 0.
\]
So we assume that $\ell=m-1$. Now $\gamma^{-1}U\cap \Span X = \Span Y_{m}$ and hence $\dim_\mathbb{R}W=n$. Hence $c_W$ is a single point, and so $\widetilde z$ is constant. For $0<\varepsilon'<\varepsilon(\widetilde z)$ let
\[
    \widetilde C_{\varepsilon'} =\{ x\in\widetilde C_\varepsilon: |v(x)|<\varepsilon'\}.
\]
Then $(\widetilde C_\varepsilon-\widetilde C_{\varepsilon'})\times [0,\delta]$ does not intersect $c(X)$ so Proposition \ref{prp:psv} implies that
\[
    \lim_{t\to\infty}\int_{\widetilde C_\varepsilon\times [0,\delta]} s^\ast\varphi^0(Xat^{\frac{1}{2}}) = \lim_{t\to\infty}\int_{\widetilde C_{\varepsilon'}\times [0,\delta]} s^\ast\varphi^0(Xat^{\frac{1}{2}}).
\]
Now we apply a variation on Lemma \ref{lem:tfb} for $(\gamma^{-1}U)^\perp$. As in the proof of said lemma, the map
\[
    c_{\gamma^{ - 1}U}\to (\gamma^{ - 1}U)^\perp \cap W
\]
given by taking orthogonal projection inside $V$, is a diffeomorphism. In addition the linear map 
\[
    (\gamma^{ - 1}U)^\perp \cap W\to \mathbb{R}^\ell, \hspace{1cm} ((Y_1, -),\ldots,(Y_\ell, -))
\]
is an isomorphism, so $(v_1,\ldots,v_\ell)$ restricts to coordinates on $c_{\gamma^{-1}U}$. Finally, if we pick $\varepsilon'<\varepsilon/3$ then 

\begin{align*}
    \lim_{t\to\infty}
        \int_{\widetilde C_{\varepsilon'}\times [0,\delta]} s^\ast\varphi^0(Xat^{\frac{1}{2}}) 
    &=
    \lim_{t\to\infty}
    \left(
        \pm\int_{B_{\varepsilon'}^{m - 1}}\int_0^{\delta} t^{\frac{m}{2}} e^{ - \pi t(|v|^2 + y^2)} dv dy 
    \right)
    \\
    &=
    \pm\int_{\mathbb{R}^{m - 1}}\int_0^\infty e^{ - \pi(|v|^2 + y^2)} dv dy = \pm\frac{1}{2}.
\end{align*}
\qed

\section{Non-holomorphic series}\label{sec:4}

Let $T=\Sym_2(F)$, $v\in\Sym_2(\mathbb{R})_{+}$, and $\varphi_f\in \mathscr{S}(\mathbf{V}(\mathbb{A}_{f})^m)^\Gamma$. For a compact special cycle of general type $c$, we let 
\[
    \beta(T,\varphi_f,v;c) : = \int_1^\infty\int_c \Theta(T,\varphi_f,\psi^0,vt) \frac{dt}{t}.
\]
We will see shortly that the outer integral indeed converges absolutely. Before then, let us also define the associated '$q$-series'. Let $\mathfrak{H}_m\subset\Sym_m(\mathbb{C})$ denote genus $m$ Siegel space. For $\tau = (\tau^{(1)},\ldots,\tau^{(r)})\in \mathfrak{H}_m^{r}$ we let $u^{(i)}$ and $v^{(i)}$ denote the real and imaginary part of $\tau^{(i)}$ resp., and we use the short-hand notation $q^T := \exp(2\pi i \sum_{i=1}^r\tr(\tau^{(i)} T^{(i)}))$. We let
\begin{equation}\label{eq:dgy}
    G^\ast(\tau,\varphi_f;c) : = \sum_{T\in\Sym_m(F)}\beta(T,\varphi_f,v^{(1)};c) q^T.
\end{equation}
The main result of this section is
\begin{proposition}\label{prp:bgc}
    $\beta$ and $G^\ast$ converge absolutely.
\end{proposition}
Before proving Proposition \ref{prp:bgc} we make some preliminary calculations.

\subsection{Orbital integrals}

For $a_1,\ldots,a_n,b>0$ define 
\[
    \mathfrak{I}(a_1,\ldots,a_n,b) : = \frac{1}{\pi }\int_0^\infty \frac{x^{ - 1/2}e^{ - \pi bx}}{(1 + x )\prod_{i = 1}^{n} (1 + x /a_i)^{1/2}} dx.
\]
The normalization factor $1/\pi$ gives the convenient inequality
\begin{equation}\label{eq:Ini}
    0 < \mathfrak{I}(a_1,\ldots,a_n,b) < \frac{1}{\pi}\int_{0}^\infty \frac{x^{ - 1 /2}}{1 + x} dx = 1.
\end{equation}
For $X\in V^{n+1}$, we let 
\[
    \varepsilon(X) : = \begin{cases} 
        + 1 & \text{if $X$ is a positively oriented basis for $V$;}
        \\ 
        -1 & \text{if $X$ is a negatively oriented basis for $V$;}
        \\
        0 &\text{if $X$ is not a basis.}
    \end{cases} 
\]

\begin{lemma}\label{lem:ewa}
    Let $X\in V^{n+1}$ be a basis, and let $\lambda_1,\ldots,\lambda_n,-\mu$ denote the eigenvalues (with multiplicity) of $4Q(X)$, with $\lambda_1,\ldots,\lambda_n,\mu>0$. Then 
    \[
        \int_1^\infty
        \int_\mathcal{H} 
        \psi^0(Xt^{\frac{1}{2}})\frac{dt}{t} 
        =
        \varepsilon(X)
        \frac{
            e^{
                -\pi \mu
            }
        }
        {
            \prod_{i = 1}^n 
            (
                1 + \mu/ \lambda_i
            )
        } 
        \mathfrak{I}\left(
            \frac{\lambda_1}{\mu} + 1,
            \ldots,
            \frac{\lambda_n}{\mu} + 1,
            \mu
        \right)
        .
    \]
\end{lemma}
\begin{proof}
By Lemma \ref{lem:bpl}(b) we may assume without loss of generality that $Q(X)$ is diagonal, so that
\[
    Q(X) = \frac{1}{4}
    \begin{pmatrix} 
        \lambda_1 & & &
        \\
        & \ddots & &\\
        &&\lambda_n &
        \\
        &&&- \mu
    \end{pmatrix}.
\]
For notational convenience let us denote $u_i=u(X_i)$. 
Lemma \ref{lem:tfb} applied for the tuple $(X_1,\ldots,X_n)$ implies that $u_1,\ldots,u_n$ provide a diffeomorphism $\mathcal{H}\cong \mathbb{R}^n$ with orientation given by $\varepsilon(X)$. We have
\[
    z = 2\left(\sum_{i = 1}^n \frac{u_i}{\lambda_i}X_i - \frac{u_{n + 1}}{\mu}X_{n + 1}\right),
\]
from which we compute 
\begin{equation}\label{eq:urf}
    - 1 = Q(z) = \sum_{i = 1}^n \frac{u_i^2}{\lambda_i} - \frac{u_{n + 1}^2}{\mu},
\end{equation}
and
\begin{equation}\label{eq:urd}
    \frac{u_{n + 1}}{\mu}du_{n + 1} =\sum_{i = 1}^n\frac{u_i}{\lambda_i}du_i.
\end{equation}
Using \eqref{eq:urd} and then \eqref{eq:urf} we obtain 
\begin{align*}
    \sum_{i = 1}^{n + 1} ( - 1)^{i - 1} du_1&\wedge\cdots \wedge \widehat{d}u_i\cdots\wedge du_{n + 1} 
    \\
    &=
    \left(
        \sum_{i = 1}^n \frac{( - 1)^{n + 1}u_i^2\mu}{u_{n + 1}\lambda_i}+( - 1)^n u_{n + 1} 
    \right)du_1\wedge \cdots\wedge du_n
    \\
    &=
    ( - 1)^n\frac{\mu}{u_{n + 1}}du_1\wedge \cdots\wedge du_n
    .
\end{align*}
Using \eqref{eq:urf} we also got 
\[
    \sum_{i = 1}^{n + 1} u_i^2 = \mu +\sum_{i = 1}^n \left(1 +\frac{\mu}{\lambda_i}\right)u_i^2 .
\]
Thus
\begin{align*}
    \int_1^\infty&\int_\mathcal{H} 
    \psi^0(Xt^{\frac{1}{2}})\frac{dt}{t} 
    \\
    &= \varepsilon(X)
    \sqrt{\mu}
    \int_1^\infty 
    \int_{\mathbb{R}^n} 
    t^{\frac{n + 1}{2}}
    \frac{
        \exp( - \pi t \left(\mu +\sum_{i = 1}^n \left(1 + \frac{\mu}{\lambda_i}\right)u_i^2\right))
    }
    {
        \left(
            1 +\sum_{i = 1}^n \frac{u_i^2}{\lambda_i}
        \right)^{1/2}
    } 
    du_1\cdots du_n\frac{dt}{t}
    \\
    &= \varepsilon(X)
    \frac{\mu^{1/2}}{\prod_{i = 1}^n (1 + \mu /\lambda_i)} 
    \int_1^\infty 
    \int_{\mathbb{R}^n} 
    e^{ -\pi t \mu}
    \frac{
        \exp( - \pi \sum_{i = 1}^n u_i^2)
    }
    {
        \left(
            t +\sum_{i = 1}^n \frac{u_i^2}{\lambda_i + \mu}
        \right)^{1/2}
    } 
    dudt
\end{align*}
It remains to verify the integral identity 
\begin{equation}\label{eq:aii}
    b^{1 /2}e^{ \pi b}\int_1^\infty 
    \int_{\mathbb{R}^n} 
    e^{ -\pi bt}
    \frac{
        \exp( - \pi \sum_{i = 1}^n u_i^2)
    }
    {
        \left(
            t+\sum_{i = 1}^n \frac{u_i^2}{A_i}
        \right)^{1/2}
    } 
    dudt
    = 
    \mathfrak{I}\left(
        \frac{A_1}{b},\ldots,\frac{A_n}{b},b
    \right)
    .
\end{equation}
To see this, apply the identity $C^{ -1/2}=\pi^{ - 1/2}\int_0^\infty x^{ - 1/2}e^{- Cx}dx$ to $C= t+\sum_{i = 1}^n \frac{u_i^2}{A_i}$ and then integrate out the $u$ and $t$ variable of the resulting $(n+2)$-fold integral using Fubini's theorem.

Finally, the first equality also follows from Fubini's theorem.
\end{proof}

\begin{remark}
    By Fubini's theorem
    \[
        \int_1^\infty
        \int_\mathcal{H} 
        \psi^0(Xt^{\frac{1}{2}})\frac{dt}{t} 
        =-\int_\mathcal{H} 
        \widetilde\psi^0(X) .
    \]
\end{remark}

\subsection{Proof of Proposition \ref{prp:bgc}}

We first consider the case where $m=n+1$ and $c=M$. Fix $a\in\GL_m^+(\mathbb{A}_\infty)$ with $\tp{a}a=v$. We have 
\[
    \beta(T,\varphi_f,v^{(1)};c) = \sum_{\substack{X\in\mathbf{V}(F)^m /\Gamma \\ Q(X) = T}} \varphi_f(X) \int_1^\infty \int_M \alpha(X a^{(1)}t^{\frac{1}{2}},\psi^0) \frac{dt}{t}.
\]
For $T\in \Sym_m(\mathbb{R})$ with eigenvalues $\lambda_1\ldots,\lambda_m$ let $\nm{T}_1=\sum_{i=1}^m|\lambda_i|$, and for $T\in \Sym_m(\mathbb{A}_\infty)$ let $\nm{T}_1 = \sum_{i=1}^r \nm{T^{(i)}}_1$.

By unfolding and using Lemma \ref{lem:ewa} and \eqref{eq:Ini}
\begin{align*}
    \left\vert
        \int_1^\infty \int_M \alpha(Xa^{(1)}t^{\frac{1}{2}},\psi^0) \frac{dt}{t} 
    \right\vert
    &= 
    \left\vert
        \int_1^\infty \int_\mathcal{H} \psi^0(Xa^{(1)}t^{\frac{1}{2}})\frac{dt}{t}
    \right\vert
    \\
    &< \exp( - 2\pi ( || Q(Xa)||_1-\tr Q(Xa)))
\end{align*}

\begin{lemma}\label{lem:ail}
    There is $C>0$ such that
    \[
    \sum_{\substack{X\in\mathbf{V}(F)^m /\Gamma \\ Q(X) = T}} \vert\varphi_f(X) \vert \ll_{\varphi_f,\Gamma} |N_{F / \mathbb{Q}}(\det T)|^C. 
\]
\end{lemma}
\begin{proof}
It is enough to consider the case in which $\varphi_f$ is the indicator function of a free $\mathcal{O}_F$-lattice $\mathbf{L}_0\subset \mathbf{V}$. Let us fix a basis $X_0\in \mathbf{L}_0$. For any $X$ with $\varphi_f(X)\neq0$, the sublattice spanned by $X$ has index $|N_{F/\mathbb{Q}}(Q(X))| / |N_{F/\mathbb{Q}}(Q(X_0))|$ in $\mathbf{L}_0$, and the number of sublattices of a given index $n$ is bounded by $n^{C_1}$, $C_1>0$. Therefore it suffices to prove the same estimate when the summation is restricted to $X$ spanning a single fixed sublattice, say $\mathbf{L}\subset \mathbf{L}_0$. Now any two such $X$ differ by an element of $\OO(\mathbf{L})$, and 
\[
    [\OO(\mathbf{L}_0):(\OO(\mathbf{L})\cap \OO(\mathbf{L}_0))] 
    \leq 
    [\GL(\mathbf{L}_0):(\GL(\mathbf{L})\cap \GL(\mathbf{L}_0))] 
    \ll_{\mathbf{L}}
    [\mathbf{L}_0:\mathbf{L}]^{C_2}
\]
for some $C_2>0$.

\end{proof}
We obtain a bound by
\[
    \sum_{T\in \alpha^{ - 1}\Sym_m(\mathcal{O}_F)} |N_{F / \mathbb{Q}}(\det T)|^C \exp( - 2\pi \nm{\tp{a}Ta}_1)
\]
where $\alpha\in\mathcal{O}_F$. Since $T\mapsto \nm{\tp{a}Ta}_1$ is a norm and $T\mapsto N_{F / \mathbb{Q}}(\det T)$ is a polynomial function on $\Sym_m(\mathbb{A}_\infty)$, and $\alpha^{-1}\Sym_m(\mathcal{O}_F)$ is a lattice inside $\Sym_m(\mathbb{A}_\infty)$, we are done.

Next we deduce the general case from this by a see-saw argument. Suppose now that $c$ comes from $\mathbf{V}' / F'$. Let $\mathbb{A}_f'$ denote the finite adeles of $F'$. As in \cite{christensen2024linking}, we may reduce to the case in which $\varphi_f = \varphi_f'\otimes \varphi_f''$, with $\varphi_f'\in\mathscr{S}(\mathbf{V}'(\mathbb{A}_{f}')^m)$ and $\varphi_f''\in\mathscr{S}(\mathbf{U}(\mathbb{A}_{f})^m)$ and we may assume that $\Gamma'$ fixes $\mathbf{U}$. Denote by $\Theta(T',\varphi_f',\psi^0,vt)_{\mathbf{V}'}$ and $G^\ast(\tau,\varphi_f';c)_{\mathbf{V}'}$ the same construction as in \eqref{eq:TTf} and \eqref{eq:dgy}, respectively, for $\mathbf{V}'$, with $\tau$ embedded in $\mathfrak{H}_m^{rs}$ by the 'diagonal map'. From Lemma \ref{lem:fos} we obtain 
\[
    \Theta(T,\varphi_f,\psi^0,vt) |_{\mathcal{D}'} = \sum_{\substack{T'\in\Sym_m(F') \\ T''\in\Sym_m(F) \\ T =\tr_{F'/F}(T') + T ''}}\Theta(T',\varphi_f',\psi^0,vt)_{\mathbf{V}'} \;\sum_{\substack{X\in \mathbf{U}(F)^m\\ Q(X) = T''}} \varphi_f''(X).
\]
This implies immediately that $\beta$ converges absolutely. We further obtain 
\[
    G^\ast(\tau,\varphi_f;c) = G^\ast(\tau,\varphi_f';c)_{\mathbf{V}'} \;\sum_{X\in\mathbf{U}(F)^m}\varphi_f''(X) q^{Q(X)}
\]
as a formal identity. Since $G^\ast(\tau,\varphi_f';c)_{\mathbf{V}'}$ converges absolutely this implies the convergence of $G^\ast(\tau,\varphi_f;c)$, and that the identity above holds not just formally.


\section{Modularity}\label{sec:5}

For simplicity we assume in this section that $n$ is odd.\footnote{In the case that $n$ is even everything should carry over, if one uses the metaplectic double cover}

\subsection{The Weil Representation}\label{sec:5.1}

For this subsection we allow $F$ to be any number field and $\mathbf{V}$ to be any non-degenerate quadratic space over $F$ of even dimension. We let $d\in F^\times / (F^\times)^2$ denote the determinant of $\mathbf{V}$.

If $A$ is either the completion of $F$ at any place or one of the rings $\mathbb{A}$, $\mathbb{A}_f$, or $\mathbb{A}_\infty$, we let $\omega$ denote the Weil representation of $\Sp_m(A)$ on $\mathscr{S}(\mathbf{V}(A)^m)$ with respect to the standard additive character of $A$. 

\subsubsection*{Local case}
Let $v$ be a place of $F$. We denote by $(,)_v$ the Hilbert symbol of $F_v$ and define a quadratic character 
\[
    \chi_v(a) = (( - 1)^{\frac{\dim \mathbf{V}}{2}}d,a)_v, \hspace{1cm} a\in F^\times.
\]
For $a\in \GL_m(F_v)$, $b\in\Sym_m(F_v)$, $\varphi\in\mathscr{S}(\mathbf{V}(F_v)^m)$, and $X\in \mathbf{V}(F_v)^m$, we have the standard formulas \cite{MR1286835}
\begin{align*}
    (\omega(\begin{smallmatrix}a & \\ & \tp{a}^{ - 1}\end{smallmatrix})\varphi)(X) = \;& \chi_v(\det a) |\det a|^{\frac{\dim \mathbf{V}}{2}} \varphi(Xa), 
    \\
    (\omega(\begin{smallmatrix} 1 & b \\ & 1\end{smallmatrix})\varphi)(X) = \; & \psi(\tr(b Q(X)))\varphi(X),
    \\
    (\omega((\begin{smallmatrix} & - 1 \\ 1 & \end{smallmatrix} ))\varphi)(X) = \;& \gamma_v^{ - m} \int_{\mathbf{V}(F_v)^m} \varphi(x)\psi( - \tr((X,x)))dx.
\end{align*}
Here $\gamma_v$ is a certain fourth root of unity, and $dx$ is the self-dual Haar measure with respect to $\tr((\;,\;))$.

Together, these formulas determine $\omega$ uniquely, since every symplectic matrix over $F_v$ is a product of upper and lower triangular symplectic matrices \cite{MR244257}.

\subsubsection*{Global case}

We fix an 
$\mathcal{O}$-lattice $\mathbf{L}\subset \mathbf{V}(F)$ such that $Q(\mathbf{L})\subset \mathfrak{d}^{-1}$. We define the dual lattice and discriminant group by 
\[
    \mathbf{L}^\# := \{X\in \mathbf{V}(F):(X,\mathbf{L})\subset \mathfrak{d}^{ - 1}\} \hspace{1cm} 
    \mathbf{D}:= \mathbf{L}^\# / \mathbf{L}
\]
We define the level of $\mathbf{L}$ to be the ideal
\[
    \mathfrak{n}:=\{x\in \mathcal{O}: xQ(\mathbf{L}^\#) \subset \mathfrak{d}^{ - 1}\}.
\]
Let us enforce the convention that a subscript $v$ is short-hand for $-\otimes \mathcal{O}_v$, and a hat is short-hand for $-\otimes \widehat{\mathcal{O}}$. So for example, $\mathfrak{d}_v$ is the different ideal of $F_v$; and $\mathbf{L}_v^\#$, resp. $\widehat{\mathbf{L}}^\#$, is the similarly defined dual lattice of $\mathbf{L}_v$, resp. $\widehat{\mathbf{L}}$.

We obtain canonical isomorphisms
\[
    \mathbf{D} \cong \widehat{\mathbf{L}}^\# / \widehat{\mathbf{L}} \cong \prod_{\text{$v$ finite}} \mathbf{L}_v^\# / \mathbf{L}_v.
\]
We identify $\mathbb{C}[\mathbf{D}^m]^\vee$ with the set of Schwartz functions supported on $(\widehat{\mathbf{L}}^\#)^m$ and constant on cosets of $\widehat{\mathbf{L}}^m$. Likewise we identify $\mathbb{C}[(\mathbf{L}_v^\# / \mathbf{L}_v)^m]$ with the set of Schwartz functions supported on $(\mathbf{L}_v^\#)^m$ and constant on cosets of $\mathbf{L}_v^m$. There is then a commutative square
\begin{equation}\label{eq:wrc}
\begin{tikzcd}
    \mathbb{C}[\mathbf{D}^m]^\vee \ar[d,hook] \ar[r,"\cong"] 
    &
    \bigotimes_{v} \mathbb{C}[(\mathbf{L}_v^\# / \mathbf{L}_v)^m]^\vee
    \ar[d,hook] 
    \\
    \mathscr{S}(\mathbf{V}(\mathbb{A}_f)^m) \ar[r,"\cong"'] 
    &
    \bigotimes'_v \mathscr{S}(\mathbf{V}(F_v)^m) 
\end{tikzcd}.
\end{equation}
\begin{lemma}
    With notation as above:
    \begin{enumerate}[label=(\roman*)]
        \item the restriction of $\omega$ to $\Sp_m(\widehat{\mathcal{O}})$ preserves the subspace $\mathbb{C}[\mathbf{D}^m]^\vee\subset \mathscr{S}(\mathbf{V}(\mathbb{A}_f)^m)$;
        \item the restriction of $\omega$ to $\ker(\Sp_m(\widehat{\mathcal{O}})\to \Sp_m(\mathcal{O} / \mathfrak{n}))$ acts trivially on $\mathbb{C}[\mathbf{D}^m]^\vee$.
    \end{enumerate}
\end{lemma}
\begin{proof}
We prove the analogous local statements, which by virtue of \eqref{eq:wrc} implies the global statement in the lemma.

(i): It is clear that the upper triangular matrices $(\begin{smallmatrix}a & \\ & \tp{a}^{ - 1}\end{smallmatrix})$, $a\in \GL_m(\mathcal{O}_v)$, and $(\begin{smallmatrix} 1 & b \\ & 1\end{smallmatrix})$, $b\in\Sym_m(\mathcal{O}_v)$, preserve $\mathbb{C}[(\mathbf{L}_v^\# / \mathbf{L}_v)^m]^\vee$. For the Weyl element, it is easy to see that if $\varphi$ is constant on cosets of $\mathbf{L}^m$ then $\omega((\begin{smallmatrix} & - 1 \\ 1 & \end{smallmatrix} ))\varphi$ is supported on $(\mathbf{L}^\#)^m$, and if $\varphi$ is supported on $(\mathbf{L}^\#)^m$ then $\omega((\begin{smallmatrix} & - 1 \\ 1 & \end{smallmatrix} ))\varphi$ is constant on cosets of $\mathbf{L}^m$. These elements generate $\Sp_m(\mathcal{O}_F)$ \cite{MR244257}, so the statements follows.

(ii): Suppose that $a\equiv  1 \bmod \mathfrak{n}_v$. To show that $(\begin{smallmatrix}a & \\ & \tp{a}^{ - 1}\end{smallmatrix})$ acts trivially, we check that $a-1$ acts trivially on $\mathbf{L}_v^\# / \mathbf{L}_v$, and that $\chi_v(\det a)=1$. 

The first part follows from the fact that $\mathbf{L}_v$ also is the dual lattice to $\mathbf{L}_v^\#$. The second part requires some case work. First let us suppose that the residue characteristic of $F_v$ is odd. In that case, if $\mathfrak{n}_v=\mathcal{O}_v$ then $\mathbf{L}_v$ must be unimodular and $(-1)^{\frac{\dim\mathbf{V}}{2}}d=1$. On the other hand, if $\mathfrak{n}_v\neq \mathcal{O}_v$ then $\det a$ must be a square. Hence we are left with the case that the residue characteristic is $2$. Now by adapting the argument in \cite[Chapter 15, 4.4.]{MR1662447} it follows that $(\mathbf{L},Q)$ is isomorphic to a direct sum of rank $1$ and rank $2$ quadratic lattices
\[
    q_i (x) =D_i x^2, \hspace{1cm} q_j(x,y) =A_jx^2 + B_jxy + C_jy^2,
\]
where $D_i,A_j,B_j,C_j\in \mathfrak{d}_v^{-1}$. Let $D_j = B_j^2-4A_jC_j$. The levels of $q_i$ and $q_j$ are $4D_i\mathfrak{d}$ and $D_j\mathfrak{d}^2$, respectively, and their discriminants are $2D_i$ and $-D_j$, respectively. Thus
\[
    (-1)^{\frac{\dim\mathbf{V}}{2}}d = \prod_i D_i \prod_j D_j,\hspace{1cm} \mathfrak{n}_v = \bigcap_i (4D_i\mathfrak{d}) \cap\bigcap_j (D_j\mathfrak{d}^2).
\]
If $\mathfrak{n}_v\subsetneqq 4\mathcal{O}_v$, then $\det(a)$ must be a square, so it suffices to consider the case in which $4\mathcal{O}_v\subset\mathfrak{n}_v$. Thus $D_i\mathfrak{d}_v = 4\mathcal{O}_v$ for all $i$, and $4\mathcal{O}_v\subset D_j\mathfrak{d}_v^2$ for all $j$. Since there is an even number of the $q_i$'s, we see that $d\in (F_v^\times)^2\mathcal{O}_v^\times$. If $\mathfrak{n}_v = 4\mathcal{O}_v$, this would imply that $\chi_v(\det a)=1$ by \cite[Proposition 21.]{MR4864833}. Thus there are none of the $q_i$'s, and $4\mathcal{O}_v\subsetneqq D_j\mathfrak{d}_v^2$ for all $j$. This implies that $2\notin B_j\mathfrak{d}_v$, and hence 
\[
    D_j \in B_j^2(1+4B_j^{-2}\mathfrak{d}_v^{-2}) = B_j^2 (1+4 D_j^{ - 1}\mathfrak{d}_v^{ - 1})
    \subset
    B_j^2 (1+4 \mathfrak{n}_v^{ - 1}),
\]
for all $j$. It follows that $\chi_v(\det a)=1$ again by \cite[Proposition 21.]{MR4864833}

If $b\equiv 0 \bmod \mathfrak{n}_v$, it follows from the definition of $\mathfrak{n}_v$ that $(\begin{smallmatrix}1 & b\\ & 1\end{smallmatrix})$ acts trivially. The congruence subgroup $\ker (\Sp_m(\mathcal{O}_v)\to \Sp_m(\mathcal{O}_v / \mathfrak{n}_v))$ is generated by upper and lower triangular matrices \cite[Proposition 13.2.]{MR244257}, so (ii) finally follows from these observations.
\end{proof}

Thus the restriction of $\omega$ determines a representation of $\Sp_m(\mathcal{O})$ on $\mathbb{C}[\mathbf{D}^m]^\vee$, which we will denote $\rho_\mathbf{L}^\vee$. We denote also by $\rho_\mathbf{L}:\Sp_m(\mathcal{O})\to\Aut \mathbb{C}[\mathbf{D}^m]$ the dual representation of $\rho^\vee$. 

One can also give explicit formulas for $\rho_\mathbf{L}$. For $\mu\in \mathbf{D}^m$, let $\varphi_{f,\mu}\in\mathscr{S}(\mathbf{V}(\mathbb{A}_f)^m)$ denote the indicator function of the set $\{X\in (\widehat{\mathbf{L}}^\#)^m : X\equiv \mu \bmod \widehat{\mathbf{L}}^m\}$, and let $\mathfrak{e}_\mu\in\mathbb{C}[\mathbf{D}^m]$ denote the corresponding basis element. Unraveling the definition gives that 
\[
    \rho_\mathbf{L}(\gamma)\mathfrak{e}_\mu = \sum_{\nu\in\mathbf{D}^m} (\omega(\gamma)^{ - 1}\varphi_{f,\nu})(\mu)\mathfrak{e}_\nu, \hspace{1cm} \gamma\in \Sp_m(\mathcal{O}).
\]
from this it is easy to verify explicit formulas for $\rho_\mathbf{L}$. Define a quadratic character $\chi:(\mathcal{O} / \mathfrak{n})^\times\to \mathbb{C}^\times$ by 
\[
    \chi(a) = \prod_{v,\; \mathfrak{n}_v\neq\mathcal{O}_v} \chi_v(a)
\]
and a fourth root of unity by $\gamma := \prod_{i = 1}^r \gamma_{\sigma_i}$, where $\sigma_1,\ldots,\sigma_r$ denotes the infinite places of $F$. Then for $a\in \GL_m(\mathcal{O} / \mathfrak{n})$ and $b\in\Sym_m(\mathcal{O} / \mathfrak{n})$
\begin{align*}
    \rho_\mathbf{L}(\begin{smallmatrix}
        a & \\ & \tp{a^{ - 1}}
    \end{smallmatrix})\mathfrak{e}_\mu 
    &= 
    \chi(a) \mathfrak{e}_{\mu a^{ - 1}},
    \\
    \rho_\mathbf{L}(\begin{smallmatrix}
        1 & b \\ & 1
    \end{smallmatrix})\mathfrak{e}_\mu 
    &= 
    e^{2\pi i \tr_{F / \mathbb{Q}}(\tr(bQ(\mu))) }\mathfrak{e}_{\mu},
    \\
    \rho_\mathbf{L}(\begin{smallmatrix}
        & - 1\\ 1 &
    \end{smallmatrix})\mathfrak{e}_\mu 
    &= 
    \left(\gamma\sqrt{|\mathbf{D}|}\right)^{ - m} \sum_{\nu\in\mathbf{D}^m} 
    e^{ -2\pi i \tr_{F / \mathbb{Q}}\tr(\mu,\nu)} \mathfrak{e}_\nu.
\end{align*}

\subsection{Modular Generating Series}

We return to the setup of Sections \ref{sec:1}--\ref{sec:4}. For $\tau\in\mathfrak{H}_m^r$ and $g = ((\begin{smallmatrix}
    \ast & \ast \\ c^{(1)} & d^{(1)}
\end{smallmatrix}),\ldots,(\begin{smallmatrix}
    \ast & \ast \\ c^{(r)} & d^{(r)}
\end{smallmatrix}))\in \Sp_m(\mathbb{A}_\infty)$, we define the factor of automorphy
\[
    j(g,\tau) = \prod_{i = 1}^r c^{(i)}\tau^{(i)} + d^{(i)}.
\]
The standard (adelic) Poisson summation argument implies

\begin{lemma}\label{lem:tps}
    The function
    \[
        \Theta_\mathbf{L}(\tau) = \sum_{\mu\in \mathbf{D}^m} \sum_{T\in\Sym_m(F)} \Theta(T,\varphi_{f,\mu},v) q^T\mathfrak{e}_\mu
    \]
    transforms as a Hilbert Siegel modular form of parallel weight $\frac{n+1}{2}$ and level $1$ with respect to $\rho$, that is
    \[
        \Theta_\mathbf{L}(\gamma \tau) =  j(\gamma,\tau)^{\frac{n + 1}{2}} \rho_\mathbf{L}(\gamma) \Theta_\mathbf{L}(\tau)
    \]
    for all $\gamma \in\Sp_m(\mathcal{O})$.
\end{lemma}

\begin{proof}
    This is a consequence of the fact that the Schwartz form $\varphi_\infty$ is the Kudla--Millson form for $\mathbf{V}$, which has weight $\dim\mathbf{V} /2$ with respect to $\omega$, and which follows for example from \cite[(5.10)]{MR4814534} and the explicit formulas in \cite{MR3796415}.
\end{proof}

We let
\[
    \Gamma_{\mathbf{L}} := \ker(\SO(\mathbf{L})\to \Aut(\mathbf{D}))
\]
and take $\Gamma\subset \Gamma_{\mathbf{L}}\cap G^0$ to be a finite index torsion-free subgroup.

We let $c$ be a fixed union of special cycles of general type in $M$, and let $s\in H_m(M,c)$ be a class whose image in $H_{m-1}(M)$ is the fundamental class of $C$. For $\tau\in \mathfrak{H}_m^r$ and $\varphi_f\in\mathscr{S}(\mathbf{V}(\mathbb{A}_f)^m)^\Gamma$ we consider the generating series
\[
    I(\tau,\varphi_f;s) : =\sum_{T\in \Sym_m(F)_+ } \iota(T,\varphi_f;s) q^T.
\]
We now define series valued in the Weil representation $\rho_\mathbf{L}$
\begin{align*}
    I_\mathbf{L}(\tau;s) &: = \sum_{\mu\in\mathbf{D}^m} I(\tau,\varphi_{f,\mu};s)\mathfrak{e}_\mu,
    \\
    G^\ast_\mathbf{L}(\tau;c) &: = \sum_{\mu\in\mathbf{D}^m} G^\ast(\tau,\varphi_{f,\mu};c)\mathfrak{e}_\mu.
\end{align*}
From Stokes theorem and Lemma \ref{lem:bpl}(e) we have the relation
\[
    \int_\sigma \Theta_\mathbf{L}(\tau) = I_{\mathbf{L}}(\tau;s) + G_\mathbf{L}^\ast(\tau;c).
\]
Lemma \ref{lem:tps} then yields
\begin{theorem}\label{thm:a}
    The function $I_\mathbf{L}(\tau,s)+G_\mathbf{L}^\ast(\tau;c)$ transforms as a Hilbert Siegel modular form of parallel weight $\frac{n+1}{2}$ and level $1$ with respect to $\rho_\mathbf{L}$. 
\end{theorem}

For the next part, we assume that $\Gamma$ is cocompact. Let $m'=n+1-m$, and consider two tuples of vectors $Y\in V^m$ and $Y'\in V^{m'}$ which span $F$-rational subspaces of $\mathbf{V}$. We are interested in the pairing
\[
    \varsigma(Y,Y'): =\int_M \alpha(Y',\varphi^0)\wedge \alpha(Y,\widetilde{\psi}^0) 
    -( - 1)^{m}
    \int_1^\infty\int_{c(Y)} \alpha(Y't^{\frac{1}{2}},\psi^0)\frac{dt}{t}.
\]
The first term converges because of \ref{lem:tps} and the compactness of $M$. The second term is convergent due to Proposition \ref{prp:bgc}. Note that when $c(Y)$ and $c(Y')$ are disjoint, then 
\begin{equation}\label{eq:ldi}
    -\int_1^\infty\int_{c(Y)} \alpha(Y't^{\frac{1}{2}},\psi^0)\frac{dt}{t}
    =
    \int_{c(Y)} \alpha(Y',\widetilde{\psi}^0),
\end{equation}
and hence $\varsigma(Y,Y') \bmod \mathbb{Z}$ is (up to sign conventions) the natural pairing $\hat{c}(Y)\star \hat{c}(Y')$ on differential cohomology between the differential characters $\hat{c}(Y)$ and $\hat{c}(Y')$. If $c(Y)$ and $c(Y')$ have intersections, then the possible difference between $\varsigma(Y,Y')\bmod \mathbb{Z}$ and $\hat{c}(Y)\star \hat{c}(Y')$ is closely related to the possible rationality of \eqref{eq:aix}, see Remark \ref{rem:igc}.

Our second theorem exhibits the pairings above in a (non-holomorphic) modular generating series. For $v\in\Sym_m(\mathbb{R})_+$ and $a\in\GL^+_m(\mathbb{R})$ with $v= a \tp{a}$, let
\[
    \varsigma(T,\varphi_f,v;Y) : = \sum_{\substack{X\in\mathbf{V}(F)^m / \Gamma \\ Q(X) = T}} \varphi_f(X)\varsigma(Y,Xa),
\]
and for $\tau\in\mathfrak{H}_m^r$ define the non-holomorphic series 
\[
    S_\mathbf{L}(\tau;Y) := \sum_{\mu\in \mathbf{D}^m} \sum_{T\in\Sym_m(F)_+} \upsilon(T,\varphi_{f,\mu},v;Y) q^{T} \mathfrak{e}_\mu.
\]

\begin{theorem}
    The function $S_\mathbf{L}(\tau;Y)+G_\mathbf{L}^\ast(\tau;c(Y))$ transforms as a Hilbert Siegel modular form of parallel weight $\frac{n+1}{2}$ and level $1$ with respect to $\rho_\mathbf{L}$. 
\end{theorem}
\begin{proof}
This follows by Lemma \ref{lem:tps} after noting that the function in question is precisely $\int_{M}\Theta_\mathbf{L}(\tau)\wedge \alpha(Y,\widetilde{\psi}^0)$.
\end{proof}
\begin{remark}\label{rem:igc}
    Suppose that $s:(\Gamma_Y\backslash c_Y)\times \mathbb{R}\to M$ is a smooth map with $s(-,0)$ equal to $c(Y)$, and such that $s$ and $c(Y)$ are transverse. For $\varepsilon>0$ let $s_\varepsilon$ be the restriction of $s$ to $(\Gamma_Y\backslash c_Y)\times [0,\varepsilon]$, and $c_\varepsilon$ the restriction of $s$ to $(\Gamma_Y\backslash c_Y)\times\{\varepsilon\}$. Then $\hat{c}(Y)$ is also represented by the spark $\alpha(Y,\widetilde{\psi}^0)-(-1)^{mm'}\delta_{s_\varepsilon}$, where $\delta$ denotes the delta current. For generic $\varepsilon$, $c_\varepsilon$ and $c(Y')$ are disjoint, hence $\hat{c}(Y)\star \hat{c}(Y')$ is given by 
\begin{equation}\label{eq:adp}
    \int_M \alpha(Y',\varphi^0)\wedge \alpha(Y,\widetilde{\psi}^0) -( - 1)^{m}
    \int_{s_\varepsilon} \alpha(Y',\varphi^0)
    +( - 1)^{m}
    \int_{c_\varepsilon} \alpha(Y',\widetilde{\psi}^0).
\end{equation}
Now by Stokes theorem and lemma \ref{lem:bpl}(e) the difference between $\varsigma(Y,Y')$ and \eqref{eq:adp} is simply
\begin{equation}
    ( - 1)^m\lim_{t\to\infty}\int_{s_\varepsilon} \alpha(Y't^{\frac{1}{2}},\varphi^0).
\end{equation}
We saw in Section \ref{sec:3} that this type of expression encodes exactly the rational part of the linking numbers.
\end{remark}

We end with a consequence of the above ideas which we find interesting. Suppose that $(\overline{\alpha}(T,\mu,v))_{T,\mu}$ is a sequence of smooth $\mathbb{R} / \mathbb{Z}$-valued functions of the variable $v\in\Sym_m(\mathbb{R})_+$. Let us say that the series $\overline{F}(\tau)=\sum_{\mu}\sum_T \overline{\alpha}(T,\mu,v) q^T \mathfrak{e}_\mu$ is a non-holomorphic Hilbert Siegel modular form modulo $\mathbb{Z}$ if there exists real-valued functions $\alpha(T,\mu,v)$ which lift $\overline{\alpha}(T,\mu,v)$ and such that
\[
    F(\tau) =\sum_{\mu\in \mathbf{D}^m}\sum_{T\in\Sym_m(F)_ +} \alpha(T,\mu,v) q^T \mathfrak{e}_\mu
\]
converges for all $\tau$ and transforms as a Hilbert Siegel modular form. 

\begin{corollary}\label{cor:mzm}
    Suppose that $Y_1,\ldots,Y_\ell$ are such that the cycle $c=c(Y_1)+\ldots+c(Y_\ell)$ is null-homologous in $M$, and let $\hat{c}=\hat{c}(Y_1)+\ldots+\hat{c}(Y_\ell)$. Then the function 
    \[
        \sum_{\mu\in \mathbf{D}^m}\sum_{T\in\Sym_m(F)_+} [\hat{c}(T,\varphi_{f,\mu},v)\star \hat{c}] q^T \mathfrak{e}_\mu
    \]
    is a non-holomorphic Hilbert Siegel modular forms modulo $\mathbb{Z}$ of level $1$ with respect to $\rho_\mathbf{L}$.
\end{corollary}

Actually, it is not so important that ${c}$ is a special cycle in the above corollary. If $\hat{c}$ is any differential character whose associated class in $H_{m}(M,\mathbb{R})$ is trivial, then we may represent $\hat{c}'$ by smooth form $\alpha$ and hence $\hat{c}(T,\varphi_{f,\mu},v)\star \hat{c} = \int_M \Theta(T,\varphi_f,\varphi^0,v)\wedge \alpha \bmod \mathbb{Z}$. It seems a quite perplexing question to us, whether or not some version of  Corollary \ref{cor:mzm} could also hold when $c$ is not null-homologous.

\bibliographystyle{alpha}
\bibliography{bibliography}

\end{document}